\documentclass[12pt]{article}
\usepackage{amsmath,amsthm,amssymb,euscript,verbatim}
\newtheorem{theorem}{Theorem}[section]

\newtheorem{corollary}{Corollary}[section]
\newtheorem{remark}{Remark}[section]

\newtheorem{conjecture}{Conjecture}[section]

\usepackage{makeidx}         
\usepackage{graphicx}        
\usepackage{multicol}        
\usepackage[bottom]{footmisc}

\makeindex             
\usepackage{latexsym}
\usepackage{epic,eepic,pstricks}

\begin{document}
\title
{\bf   Using Lie Sphere Geometry to Study Dupin Hypersurfaces in ${\bf R}^n$}
\author
{Thomas E. Cecil}
\maketitle

\begin{abstract}
A hypersurface $M$ in ${\bf R}^n$ or $S^n$ is said to be {\em Dupin} if along
each curvature surface, the corresponding principal curvature is constant.  A Dupin
hypersurface is said to be {\em proper Dupin}
if each principal curvature has constant multiplicity on $M$, i.e., the number
of distinct principal curvatures is constant on $M$. 
The notions of Dupin and proper Dupin hypersurfaces in ${\bf R}^n$ or $S^n$ can be generalized
to the setting of Lie sphere geometry, and these properties are easily seen to be invariant under Lie sphere transformations.  This makes
Lie sphere geometry an effective setting for the study of Dupin hypersurfaces, and many classifications of proper Dupin hypersurfaces have been obtained up to Lie sphere transformations.
In these notes, we give a detailed introduction to this method for studying Dupin hypersurfaces in ${\bf R}^n$ or $S^n$,
including proofs of several fundamental results. 
 \end{abstract}

\section{Introduction}
\label{sec:1}
A hypersurface $M$ in ${\bf R}^n$ or $S^n$ is said to be {\em Dupin} if along
each curvature surface, the corresponding principal curvature is constant.  A Dupin
hypersurface is said to be {\em proper Dupin}
if each principal curvature has constant multiplicity on $M$, i.e., the number
of distinct principal curvatures is constant on $M$ (see Pinkall  \cite{P4}). 

The notions of Dupin and proper Dupin hypersurfaces in ${\bf R}^n$ or $S^n$ can be generalized
to the setting of Lie sphere geometry, and these properties are easily seen to be invariant under Lie sphere transformations (see Theorem \ref{thm:dupin-lie-invariant}).  This makes
Lie sphere geometry an effective setting for the study of Dupin hypersurfaces, and many classifications of proper Dupin hypersurfaces have been obtained up to Lie sphere transformations.

In these notes, we give a detailed introduction to the method for studying Dupin hypersurfaces in ${\bf R}^n$ or $S^n$ using Lie sphere geometry, including necessary and sufficient conditions for a Dupin hypersurface in $S^n$ to be equivalent to an isoparametric hypersurface in $S^n$ by a Lie sphere transformation (Theorem \ref{thm:3.5.6}).  As an application, we give a classification due to Pinkall \cite{P4} of the cyclides of Dupin in $S^n$ up to Lie sphere transformations (Theorem \ref{thm:4.3.1}).  
We also give a survey of results concerning compact proper Dupin hypersurfaces and their relationship to isoparametric hypersurfaces (see Section \ref{sec:12}).  

In 1872, Lie \cite{Lie} introduced his
geometry of oriented hyperspheres in Euclidean space ${\bf R}^n$ in the context of his work on contact
transformations (see \cite{LS}).  Lie established a bijective correspondence between the set of all {\em Lie spheres}, i.e.,
oriented hyperspheres, oriented hyperplanes and point spheres,
in ${\bf R}^n \cup \{\infty \}$, and the set of all points on the quadric hypersurface $Q^{n+1}$ in real projective space 
${\bf P}^{n+2}$ given by the equation $\langle x,x \rangle = 0$, where $\langle\  ,\  \rangle$ is
an indefinite scalar product with signature $(n+1,2)$ on ${\bf R}^{n+3}$ given by
\begin{equation}
\label{eq:lie-product}
\langle x,y \rangle = - x_1 y_1 + x_2 y_2 + \cdots + x_{n+2} y_{n+2} - x_{n+3} y_{n+3},
\end{equation} 
for $x = (x_1,\ldots, x_{n+3})$, $y = (y_1,\ldots, y_{n+3})$.

Using linear algebra, one can show that this {\em Lie quadric} $Q^{n+1}$ contains 
projective lines but no linear subspaces of ${\bf P}^{n+2}$ of higher dimension,
since the metric in equation (\ref{eq:lie-product}) has signature
$(n+1,2)$ (see \cite [p.  21]{Cec1}).  The one-parameter family of Lie spheres in
${\bf R}^n \cup \{\infty \}$ corresponding to the points on a line on $Q^{n+1}$ is called a 
{\em parabolic pencil} of spheres.
It consists of all Lie spheres in oriented contact at a certain contact element 
$(p,N)$ on ${\bf R}^n \cup \{\infty \}$, where $p$ is a point in ${\bf R}^n \cup \{\infty \}$ and $N$ is a unit
tangent vector to ${\bf R}^n \cup \{\infty \}$ at $p$.  That is, $(p,N)$ is an element of the unit tangent bundle 
of ${\bf R}^n \cup \{ \infty \}$.
In this way, Lie also established a bijective correspondence between the manifold of contact elements on 
${\bf R}^n \cup \{\infty \}$ and the manifold $\Lambda^{2n-1}$ of projective lines on the Lie quadric $Q^{n+1}$.

A {\em Lie sphere transformation} 
is a projective transformation of ${\bf P}^{n+2}$ which maps the Lie quadric $Q^{n+1}$
to itself.  In terms of the geometry of ${\bf R}^n$, a Lie sphere transformation maps Lie spheres to Lie spheres.
Furthermore, since a projective transformation maps lines to lines, a Lie sphere transformation preserves oriented contact of Lie spheres in ${\bf R}^n$.  

Let ${\bf R}^{n+3}_2$ denote ${\bf R}^{n+3}$ endowed with the 
metric $\langle \ ,\  \rangle$ in equation (\ref{eq:lie-product}), and let $O(n+1,2)$
denote the group of orthogonal transformations of ${\bf R}^{n+3}_2$.
One can show that every Lie sphere transformation is the projective transformation
induced by an orthogonal transformation, and
thus the group $G$ of Lie sphere transformations is isomorphic to the quotient group $O(n+1,2)/\{\pm I\}$
(see \cite[pp. 26--27]{Cec1}).
Furthermore, any M\"{o}bius (conformal) 
transformation of ${\bf R}^n \cup \{\infty \}$ induces a
Lie sphere transformation, and the M\"{o}bius group 
is precisely the subgroup of $G$ consisting of all Lie sphere transformations
that map point spheres to point spheres.  

The manifold $\Lambda^{2n-1}$ of projective lines on the quadric $Q^{n+1}$ has a 
contact structure, i.e., a
$1$-form $\omega$ such that $\omega \wedge (d\omega)^{n-1}$ does not vanish on $\Lambda^{2n-1}$.  The condition $\omega = 0$ defines a codimension one distribution $D$ on $\Lambda^{2n-1}$ which has 
integral submanifolds
of dimension $n-1$, but none of higher dimension.  Such an integral  submanifold 
$\lambda: M^{n-1} \rightarrow \Lambda^{2n-1}$ of $D$ of maximal dimension is called a 
{\em Legendre submanifold}.
If $\alpha$ is a Lie sphere transformation, then $\alpha$ maps lines on $Q^{n+1}$ to lines on $Q^{n+1}$, and
the map $\mu = \alpha \lambda$ is also a Legendre submanifold. The
submanifolds $\lambda$ and $\mu$ are said to be {\em Lie equivalent}.

Let $M^{n-1}$ be an oriented hypersurface in ${\bf R}^n$. Then $M^{n-1}$ naturally induces a Legendre submanifold $\lambda: M^{n-1} \rightarrow \Lambda^{2n-1}$, called the {\em Legendre lift of}
$M^{n-1}$, as we will show in these notes.  More generally, an immersed submanifold
$V$ of codimension greater than one in ${\bf R}^n$
induces a Legendre lift whose domain is the unit normal
bundle $B^{n-1}$ of $V$ in ${\bf R}^n$.  Similarly, any submanifold of the unit sphere $S^n \subset {\bf R}^{n+1}$ has a Legendre lift.
We can relate properties of  a submanifold of ${\bf R}^n$ or $S^n$ to Lie geometric properties of its Legendre lift,
and attempt to classify certain types of Legendre submanifolds up to Lie sphere transformations.
This, in turn, gives classification results for the corresponding classes of Euclidean submanifolds of ${\bf R}^n$
or $S^n$.

We next recall some basic ideas from Euclidean submanifold theory that are necessary for the study of
Dupin hypersurfaces.
For an oriented hypersurface $f:M\rightarrow {\bf R}^n$ with field of unit normal vectors $\xi$,
 the eigenvalues of the shape operator
(second fundamental form) $A$ of $M$ are called {\em principal curvatures},
and their corresponding eigenspaces
are called {\em principal spaces}.  
A submanifold $S$ of $M$ is called a {\em curvature surface} of $M$
if at each
point $x$ of $S$, the tangent space $T_xS$ is a principal space at $x$.  This generalizes the classical notion of a 
line of curvature
of a surface in ${\bf R}^3$.  
If $\kappa$ is a non-zero principal curvature of $M$ at $x$, the point 
\begin{equation}
\label{eq:focal-point}
f_\kappa (x) = f(x) + (1/\kappa) \xi (x)
\end{equation} 
is called the {\em focal point} of $M$ at $x$ determined by $\kappa$.  The 
hypersphere in ${\bf R}^n$ tangent to $M$ at $f(x)$ and centered at the focal point $f_\kappa (x)$ is called the 
{\em curvature sphere} at $x$ 
determined by $\kappa$.

It is well known that there always exists an open
dense subset $\Omega$ of $M$ on which the multiplicities of the principal curvatures are locally 
constant (see, for example, Singley \cite{Sin}).
If a principal curvature $\kappa$ has constant multiplicity $m$
in some open set $U \subset M$, then the corresponding $m$-dimensional
distribution of principal spaces is integrable, i.e., it is
an $m$-dimensional
foliation, and the leaves of this {\em principal foliation}
are curvature surfaces.  Furthermore, if the
multiplicity $m$ of $\kappa$ is greater than one, then by using the Codazzi equation, 
one can show that $\kappa$ is constant along each leaf of this principal
foliation (see, for example, \cite[p. 24]{CR8}).  This is not true, in general, if the multiplicity $m=1$.  
Analogues of these results hold for oriented hypersurfaces in the sphere $S^n$ or in real hyperbolic space $H^n$
(see, for example, \cite[pp. 9--35]{CR8}).

As mentioned earlier, a hypersurface $M$ in ${\bf R}^n$ or $S^n$ is said to be Dupin if along
each curvature surface, the corresponding principal curvature is constant.  A Dupin
hypersurface is said to be proper Dupin
if each principal curvature has constant multiplicity on $M$, i.e., the number
of distinct principal curvatures is constant on $M$ (see Pinkall  \cite{P4}). 

The notions of Dupin and proper Dupin hypersurfaces in ${\bf R}^n$ or $S^n$ can be generalized
to a class of Legendre submanifolds in Lie sphere geometry known as Dupin submanifolds
(see Section \ref{sec:7a}).
In particular, the Legendre lifts of Dupin hypersurfaces in ${\bf R}^n$ or $S^n$  are Dupin submanifolds in 
this generalized sense.
The Dupin and proper Dupin properties of such submanifolds are easily seen to be invariant under Lie sphere transformations (see Theorem \ref{thm:dupin-lie-invariant}).  

A well-known class of proper Dupin hypersurfaces consists of the cylides of Dupin in ${\bf R}^3$,
introduced by Dupin \cite{D} in 1822 (see Section \ref{sec:10}).   Dupin defined a cyclide 
to be a surface $M$ in ${\bf R}^3$ that is the envelope
of the family of spheres tangent to three fixed spheres in ${\bf R}^3$.  This is equivalent to requiring that $M$
have two distinct principal curvatures at each point, and that both focal maps of $M$ degenerate into 
curves (instead of surfaces).  
Then $M$ is the envelope of the family of curvature spheres centered along each of the focal curves.  The three fixed spheres in Dupin's definition can be chosen to be three spheres from either family of curvature spheres.

The most basic examples of cyclides of Dupin in ${\bf R}^3$ are a torus of revolution, a circular cylinder, and a
circular cone.  The proper Dupin property is easily shown to be invariant under M\"{o}bius transformations 
of ${\bf R}^3 \cup \{\infty \}$, and
it turns out that all cyclides of Dupin in ${\bf R}^3$ can be obtained from these three types of examples by inversion in a sphere in ${\bf R}^3$ (see, for example, \cite[pp. 151--166]{CR7}).  

Pinkall's paper \cite{P4} describing 
higher dimensional cyclides of Dupin
in the context of Lie sphere geometry was particularly influential.
Pinkall defined a {\em cyclide of Dupin of characteristic} $(p,q)$ to be
a proper Dupin submanifold $\lambda: M^{n-1} \rightarrow \Lambda^{2n-1}$ with two distinct curvature spheres
of respective multiplicities $p$ and $q$ at each point.
In Section \ref{sec:10} of these notes, 
we present Pinkall's \cite{P4} classification of the cyclides of Dupin 
of arbitrary dimension in ${\bf R}^n$ or $S^n$ (Theorem \ref{thm:4.3.1}), which is
obtained by using the methods of Lie sphere geometry.
 
Specifically, we show in Theorem \ref{thm:4.3.1} that any connected cyclide of Dupin of
characteristic $(p,q)$ is contained in a unique compact, connected cyclide of Dupin of
characteristic $(p,q)$.  
Furthermore, every compact, connected cyclide of Dupin of characteristic $(p,q)$ is Lie equivalent to the Legendre lift of a standard product of two spheres,
\begin{equation}
\label{eq:std-prod-p-q-1}
S^q (1/\sqrt{2}) \times S^p (1/\sqrt{2}) \subset S^n \subset {\bf R}^{q+1} \times {\bf R}^{p+1} = {\bf R}^{n+1}. 
\end{equation} 

A  standard product of two spheres in $S^n$ is an isoparametric hypersurface in $S^n$, i.e., it has
constant principal curvatures in $S^n$.  The images of isoparametric hypersurfaces in $S^n$ under stereographic projection from $S^n$ to ${\bf R}^n$ form a particularly important class of
proper Dupin hypersurfaces in ${\bf R}^n$.

Many results in the field deal with relationships between compact proper Dupin hypersurfaces 
and isoparametric hypersurfaces in spheres, including the question of which compact proper Dupin hypersurfaces are 
Lie equivalent to an isoparametric hypersurface in a sphere.  An important result in proving such
classifications is Theorem \ref{thm:3.5.6} of these notes, 
which gives for necessary and sufficient conditions for a Legendre
submanifold to be Lie equivalent to the Legendre lift of an isoparametric hypersurface in $S^n$.

Local examples of proper Dupin hypersurfaces in ${\bf R}^n$ or $S^n$ are known to be plentiful, since  Pinkall 
\cite{P4} introduced four constructions
for obtaining a proper Dupin hypersurface
$W$ in ${\bf R}^{n+m}$ from a proper Dupin hypersurface
$M$ in ${\bf R}^n$.  These constructions involve building tubes, cylinders, 
cones and surfaces of revolution
from $M$, and they are discussed in Section \ref{sec:11}. Using these constructions, Pinkall
was able to construct a proper Dupin hypersurface in Euclidean space with an arbitrary number of distinct
principal curvatures with any given multiplicities (see Theorem \ref{thm:4.1.1}).  
In general, these proper Dupin hypersurfaces obtained by using Pinkall's constructions cannot be
extended to compact Dupin hypersurfaces without losing the property that the number of distinct principal
curvatures is constant, as we discuss in Section \ref{sec:11}.

Proper Dupin hypersurfaces that are locally Lie equivalent to the end product of one of
Pinkall's constructions are said to be {\em reducible}. 
Pinkall \cite{P4} found a useful characterization of reducibility in the context of Lie sphere 
geometry, which we state and prove in Theorem \ref{thm:4.2.9}.  This theorem has been important in proving
several classifications of irreducible proper Dupin hypersurfaces that are described in Section \ref{sec:12}.

In Section \ref{sec:12}, we give a survey of results concerning compact and irreducible proper Dupin hypersurfaces, and their relationship to isoparametric hypersurfaces.  Cecil, Chi and Jensen \cite{CCJ2} showed that 
every compact proper Dupin hypersurface with more than two principal curvatures is irreducible (Theorem
\ref{thm:CCJ-2007}).
In fact, several classifications of compact proper Dupin hypersurfaces with $g \geq 3$
principal curvatures have been obtained by assuming that the hypersurface is irreducible and 
working locally in the context of Lie sphere geometry using the method of moving frames.  
(See, for example, the papers of Pinkall \cite{P1}, \cite{P3}--\cite{P4}, Cecil and Chern \cite{CC2}, Cecil and Jensen \cite{CJ2}--\cite{CJ3}, Cecil, Chi and Jensen \cite{CCJ2}, and Niebergall \cite{N1}--\cite{N2}.)  
See also \cite{Cec10} for a survey of
classifications of compact proper Dupin hypersurfaces.

These notes are based primarily on the author's book \cite{Cec1}, and 
several passages in these notes are taken directly from that book.

\section{Notation and Preliminary Results}
\label{sec:2}
In the next few sections, we review the basic setup for the sphere geometries of M\"{o}bius and Lie, and the method for studying submanifolds of ${\bf R}^n$ and $S^n$ in this context.

Let $(x,y)$ be the indefinite scalar product on the Lorentz space
${\bf R}^{n+1}_1$ defined by
\begin{equation}
\label{eq:1.1.1}
(x,y) = - x_1y_1 + x_2y_2 + \cdots + x_{n+1}y_{n+1},
\end{equation}
where $x = (x_1,\ldots,x_{n+1})$ and $y = (y_1,\ldots,y_{n+1})$.  We will call
this scalar product the {\em Lorentz metric}.
A vector $x$ is said to
be {\em spacelike}, {\em timelike} or {\em lightlike}, respectively,
depending on whether $(x,x)$ is positive, negative or zero.  We
will use this terminology even when we are using an indefinite metric of
different signature.

In Lorentz space, the set of all lightlike
vectors, given by the equation,
\begin{equation}
\label{eq:1.1.2}
x_1^2 = x_2^2 + \cdots + x_{n+1}^2,
\end{equation}
forms a cone of revolution, called the 
{\em light cone}.   Timelike vectors are
``inside the cone'' and spacelike vectors are ``outside the cone.''

If $x$ is a nonzero vector in ${\bf R}^{n+1}_1$, let $x^{\perp}$ 
denote the orthogonal complement of $x$ with respect to the Lorentz metric.
If $x$ is timelike, then the metric restricts to a positive definite
form on $x^{\perp}$, and $x^{\perp}$ intersects the light cone only at
the origin.  If $x$ is spacelike, then the metric has signature $(n-1,1)$
on $x^{\perp}$, and $x^{\perp}$ intersects the cone in a cone of
one less dimension.  If $x$ is lightlike, then $x^{\perp}$ is tangent to
the cone along the line through the origin determined by $x$.  The
metric has signature $(n-1,0)$ on this $n$-dimensional plane.

Lie sphere geometry is defined in the context of real projective space ${\bf P}^n$,
so we now briefly review some basic concepts from projective geometry.
We define an equivalence relation on ${\bf R}^{n+1} - \{0\}$ by setting
$x \simeq y$ if $x = ty$ for some nonzero real number $t$.  We denote
the equivalence class determined by a vector $x$ by $[x]$.  Projective
space ${\bf P}^n$ is the set of such equivalence classes, and it can
naturally be identified with the space of all lines through the origin
in ${\bf R}^{n+1}$.  The rectangular coordinates $(x_1, \ldots, x_{n+1})$
are called {\em homogeneous coordinates}
of the point $[x]$, and they
are only determined up to a nonzero scalar multiple.  The affine space
${\bf R}^n$ can be embedded in ${\bf P}^n$ as the complement of the hyperplane
$(x_1 = 0)$ at infinity by the map $\phi: {\bf R}^n \rightarrow
{\bf P}^n$ given by $\phi (u) = [(1,u)]$.  A scalar product on 
${\bf R}^{n+1}$, such as the Lorentz metric, determines a polar
relationship between points and hyperplanes in ${\bf P}^n$.  We
will also use the notation $x^{\perp}$ to denote the polar hyperplane
of $[x]$ in ${\bf P}^n$, and we will call $[x]$ the {\em pole} of $x^{\perp}$.

If $x$ is a non-zero lightlike vector in ${\bf R}^{n+1}_1$, then $[x]$ can be represented
by a vector of the form $(1,u)$ for $u \in {\bf R}^n$.  Then the equation
$(x,x) = 0$ for the light cone becomes $u \cdot u = 1$ (Euclidean dot product),
i.e., the equation for the unit sphere in  ${\bf R}^n$.  Hence, the set of 
points in ${\bf P}^n$ determined by lightlike vectors in ${\bf R}^{n+1}_1$
is naturally diffeomorphic to the sphere $S^{n-1}$.

\section{M\"{o}bius Geometry of Unoriented Spheres}
\label{sec:3}
As a first step toward Lie's geometry of oriented spheres, we recall the geometry of
unoriented spheres in ${\bf R}^n$ known as ``M\"{o}bius'' or
``conformal'' geometry.  We will always assume that $n \geq 2.$
In this section, we will only consider spheres and planes
of codimension one, and we will often omit the prefix ``hyper,'' from the words
``hypersphere'' and ``hyperplane.'' (See \cite[pp. 11--14]{Cec1} for more detail.)

We denote the Euclidean dot product of two vectors $u$ and $v$ in
${\bf R}^n$ by $u \cdot v$.  We first consider stereographic
projection $\sigma : {\bf R}^n \rightarrow S^n - \{P\}$, where
$S^n$ is the unit sphere in ${\bf R}^{n+1}$ given by $y \cdot y = 1$,
and $P = (-1,0,\ldots,0)$ is the south pole
of $S^n$.
The well-known formula for $\sigma (u)$ is
\begin{displaymath}
\sigma (u) = \left(\frac{1-u \cdot u}{1+ u \cdot u}, \frac{2u}{1+ u \cdot u}
\right).
\end{displaymath}
Note that $\sigma$ is sometimes referred to as ``inverse stereographic projection,'' in which case its inverse map 
from $S^n - \{P\}$ to ${\bf R}^n$ is called ``stereographic projection.''

We next embed ${\bf R}^{n+1}$ into ${\bf P}^{n+1}$ by the embedding $\phi$
mentioned in the previous section.  Thus, we have the map 
$\phi \sigma : {\bf R}^n \rightarrow {\bf P}^{n+1}$ given by
\begin{equation}
\label{eq:1.2.1}
\phi \sigma (u) = \left[ \left( 1,\frac{1-u \cdot u}{1+ u \cdot u}, \frac{2u}{1+ u \cdot u}
\right) \right] = \left[ \left( \frac{1+u \cdot u}{2},\frac{1-u \cdot u}{2},u \right) \right].
\end{equation}
Let $(z_1,\ldots,z_{n+2})$ be homogeneous coordinates on ${\bf P}^{n+1}$
and $(\ ,\ )$ the Lorentz metric on the space ${\bf R}^{n+2}_1$.  Then
$\phi \sigma ({\bf R}^n)$ is just the set of points in ${\bf P}^{n+1}$
lying on the $n$-sphere $\Sigma$ given by the equation $(z,z) = 0$,
with the exception of the {\em improper point} $[(1,-1,0,\ldots,0)]$
corresponding to the south pole $P$.  We will refer to 
the points in $\Sigma$ other than $[(1,-1,0,\ldots,0)]$ as
{\em proper points}, and will call $\Sigma$ the 
{\em M\"{o}bius sphere} or {\em M\"{o}bius space}.  At times, it
is easier to simply begin with $S^n$ rather than ${\bf R}^n$ and
thus avoid the need for the map $\sigma$ and the special
point $P$.  However, there are also advantages for beginning in ${\bf R}^n$.

The basic framework for the M\"{o}bius geometry of unoriented spheres
is as follows.  Suppose that $\xi$ is a spacelike vector in ${\bf R}^{n+2}_1$.
Then the polar hyperplane $\xi^{\perp}$ to $[\xi]$ in ${\bf P}^{n+1}$
intersects the sphere $\Sigma$ in an $(n-1)$-sphere $S^{n-1}$.  The sphere
$S^{n-1}$ is the image under $\phi \sigma$ of an $(n-1)$-sphere in ${\bf R}^n$,
unless it contains the improper point $[(1,-1,0,\ldots,0)]$, in which case
it is the image under $\phi \sigma$ of a hyperplane in ${\bf R}^n$.  Hence,
we have a bijective correspondence between the set of all spacelike points in 
${\bf P}^{n+1}$ and the set of all hyperspheres and hyperplanes in ${\bf R}^n$.

It is often useful to have specific formulas for this correspondence.
Consider the sphere in ${\bf R}^n$ with center $p$ and radius $r>0$
given by the equation
\begin{equation}
\label{eq:1.2.2}
(u-p) \cdot (u-p) = r^2.
\end{equation}
We wish to translate this into an equation involving the Lorentz metric
and the corresponding polarity relationship on ${\bf P}^{n+1}$.  A direct calculation
shows that equation (\ref{eq:1.2.2}) is equivalent to the equation
\begin{equation}
\label{eq:1.2.3}
(\xi,\phi \sigma (u)) = 0,
\end{equation}
where $\xi$ is the spacelike vector,
\begin{equation}
\label{eq:1.2.4}
\xi = \left( \frac{1+p \cdot p - r^2}{2}, \frac{1- p \cdot p +r^2}{2},p \right),
\end{equation}
and $\phi \sigma (u)$ is given by equation (\ref{eq:1.2.1}).  Thus, the point $u$ is on the sphere
given by equation (\ref{eq:1.2.2}) if and only if $\phi \sigma (u)$ lies on the polar
hyperplane of $[\xi]$.  Note that the first two coordinates
of $\xi$ satisfy $\xi_1 + \xi_2 = 1$, and that
$(\xi,\xi) = r^2$.  Although $\xi$ is only determined
up to a nonzero scalar multiple, we can conclude that 
$\eta_1 + \eta_2$ is not zero for any $\eta \simeq \xi$.

Conversely, given a spacelike point $[z]$ with $z_1 + z_2$ nonzero,
we can determine the corresponding sphere in ${\bf R}^n$ as follows.
Let $\xi = z/(z_1 + z_2)$ so that $\xi_1 + \xi_2 = 1$. Then from equation
(\ref{eq:1.2.4}), the center of the corresponding sphere is the point
$p = (\xi_3,\ldots,\xi_{n+2})$, and the radius is the square root of
$(\xi,\xi)$.

Next suppose that $\eta$ is a spacelike vector with $\eta_1 + \eta_2 =0$.
Then
\begin{displaymath}
(\eta, (1,-1,0,\ldots,0)) = 0.
\end{displaymath}
Thus, the improper point $\phi (P)$ lies on the polar hyperplane of $[\eta ]$,
and the point $[\eta ]$ corresponds to a hyperplane in ${\bf R}^n$.
Again we can find an explicit correspondence.  Consider the hyperplane in
${\bf R}^n$ given by the equation
\begin{equation}
\label{eq:1.2.5}
u \cdot N = h, \quad |N| = 1.
\end{equation}
A direct calculation shows that (\ref{eq:1.2.5}) is equivalent to the equation
\begin{equation}
\label{eq:1.2.6}
(\eta, \phi \sigma (u)) = 0, {\rm where} \ \eta = (h, -h, N).
\end{equation}
Thus, the hyperplane (\ref{eq:1.2.5}) is represented in the polarity relationship
by $[\eta ]$.  

Conversely, let $z$ be a spacelike point with
$z_1 + z_2 = 0$.  Then $(z,z) = v \cdot v$, where $v = (z_3,\ldots,z_{n+2})$.
Let $\eta = z / |v|$.  Then $\eta$ has the form (\ref{eq:1.2.6}) and
$[z]$ corresponds to the hyperplane (\ref{eq:1.2.5}).  Thus we have explicit formulas for
the bijective correspondence between the set of spacelike points in ${\bf P}^{n+1}$
and the set of hyperspheres and hyperplanes in ${\bf R}^n$.

Similarly, we can construct a bijective correspondence between the space of all hyperspheres in the unit sphere
$S^n \subset {\bf R}^{n+1}$ and the manifold of all spacelike points in ${\bf P}^{n+1}$ as follows.  The hypersphere $S$ in $S^n$ with center $p \in S^n$
and (spherical) radius $\rho, 0 < \rho < \pi$, is given by the equation 
\begin{equation}
\label{eq:1.4.1}
p \cdot y = \cos \rho, \quad 0 < \rho < \pi,
\end{equation}
for $y \in S^n$.  If we take $[z] = \phi(y) = [(1,y)]$, then
\begin{displaymath}
p \cdot y = \frac{-(z,(0,p))}{(z,e_1)},
\end{displaymath}
where $e_1 = (1,0,\ldots,0)$.
Thus equation (\ref{eq:1.4.1}) is equivalent to the equation
\begin{equation}
\label{eq:1.4.2}
(z,(\cos \rho,p)) = 0,
\end{equation}
in homogeneous coordinates in ${\bf P}^{n+1}$. Therefore,
$y$ lies on the hypersphere $S$ given by equation (\ref{eq:1.4.1}) if and only if $[z] = \phi (y)$
lies on the polar hyperplane in ${\bf P}^{n+1}$ of the spacelike point 
\begin{equation}
\label{eq:1.4.3}
[\xi] = [(\cos \rho,p)].
\end{equation}

\begin{remark}
\label{spheres-hyperbolic-space}
{\rm In these notes, we will focus on spheres in ${\bf R}^n$  or $S^n$.
See \cite[pp. 16--18]{Cec1} for a treatment of the geometry of hyperspheres in real hyperbolic space $H^n$.}
\end{remark}

Of course, the fundamental invariant of M\"{o}bius geometry is the angle.
The study of angles in this setting is quite natural, since orthogonality 
between spheres and planes in ${\bf R}^n$ can be expressed in terms of the Lorentz
metric.  Let $S_1$ and $S_2$ denote the spheres in ${\bf R}^n$ with respective
centers $p_1$ and $p_2$ and respective radii $r_1$ and $r_2$.  By the
Pythagorean Theorem, the two spheres intersect orthogonally if and only if
\begin{equation}
\label{eq:1.2.7}
|p_1 - p_2|^2 = r_1^2 + r_2^2.
\end{equation}
If these spheres correspond by equation (\ref{eq:1.2.4}) to the projective points
$[\xi_1]$ and $[\xi_2]$, respectively, then a calculation shows that equation
(\ref{eq:1.2.7}) is equivalent to the condition
\begin{equation}
\label{eq:1.2.8}
(\xi_1,\xi_2) = 0.
\end{equation}

A hyperplane $\pi$ in ${\bf R}^n$ is orthogonal to a hypersphere $S$ precisely
when $\pi$ passes through the center of $S$.  If $S$ has center $p$ and
radius $r$, and $\pi$ is given by the equation $u \cdot N = h$, then the
condition for orthogonality is just $p \cdot N = h$.  If $S$ corresponds
to $[\xi]$ as in (\ref{eq:1.2.4}) and $\pi$ corresponds to $[\eta]$ as in
(\ref{eq:1.2.6}), then this equation for orthogonality is equivalent to
$(\xi, \eta) = 0$.  Finally, if two planes $\pi_1$ and $\pi_2$ are represented
by $[\eta_1]$ and $[\eta_2]$ as in (\ref{eq:1.2.6}), then the orthogonality
condition $N_1 \cdot N_2 = 0$ is equivalent to the equation
$(\eta_1, \eta_2) = 0$.  Thus, in all cases of
hyperspheres or hyperplanes in ${\bf R}^n$, orthogonal intersection corresponds to a polar relationship in
${\bf P}^{n+1}$ given by equations (\ref{eq:1.2.3}) or (\ref{eq:1.2.6}).

We conclude this section with a discussion of M\"{o}bius transformations. Recall that a linear transformation $A \in GL(n+2)$ induces a projective transformation $P(A)$ on ${\bf P}^{n+1}$ defined by $P(A)[x] = [Ax]$.
The map $P$ is a homomorphism of $GL(n+2)$ onto the group $PGL(n+1)$ of projective transformations of 
${\bf P}^{n+1}$, and its kernel is the group of nonzero multiples of the identity transformation $I \in GL(n+2)$.

A {\em M\"{o}bius transformation} is a projective transformation $\alpha$ of ${\bf P}^{n+1}$ that preserves the
condition $(\eta, \eta) = 0$ for $[\eta] \in {\bf P}^{n+1}$, that is, $\alpha = P(A)$, where $A \in GL(n+2)$
maps lightlike vectors in ${\bf R}^{n+2}_1$ to lightlike vectors.  It can be shown
(see, for example, \cite[pp. 26--27]{Cec1}) that such a linear transformation $A$ is a nonzero scalar multiple of
a linear transformation $B \in O(n+1,1)$, the orthogonal group for the Lorentz inner product 
space ${\bf R}^{n+2}_1$.
Thus, $\alpha = P(A) = P(B)$.

The M\"{o}bius transformation $\alpha = P(B)$ induced by an orthogonal transformation $B \in O(n+1,1)$ maps
spacelike points to spacelike points in ${\bf P}^{n+1}$, and it preserves the polarity condition
$(\xi, \eta) = 0$ for any two points $[\xi]$ and $[\eta]$ in ${\bf P}^{n+1}$.  Therefore by the
correspondence given in equations (\ref{eq:1.2.3}) and (\ref{eq:1.2.6}) above, $\alpha$ maps the set of hyperspheres 
and hyperplanes in ${\bf R}^n$ to itself, and it preserves orthogonality and hence angles between hyperspheres
and hyperplanes.  A similar statement holds for the set of all hyperspheres in $S^n$.

Let $H$ denote the group of M\"{o}bius transformations and let 
\begin{equation}
\label{eq:defn-psi}
\psi: O(n+1,1) \rightarrow H
\end{equation}
be the restriction
of the map $P$ to $O(n+1,1)$.  The discussion above shows that $\psi$ is onto, and the kernel of $\psi$ is
$\{\pm I \}$, the intersection of $O(n+1,1)$ with the kernel of $P$.  Therefore, $H$ is isomorphic to the quotient
group $O(n+1,1)/ \{\pm I \}$.  

One can show that the group $H$ is generated by M\"{o}bius transformations 
induced by inversions in spheres in ${\bf R}^n$. 
This follows from the fact that the corresponding orthogonal
groups are generated by reflections 
in hyperplanes.  In fact, every orthogonal transformation on an indefinite inner product space ${\bf R}^n_k$ is
a product of at most $n$ reflections, a result due to Cartan and Dieudonn\'{e}.
(See Cartan \cite[pp. 10--12]{Car6}, Chapter 3 of Artin's book \cite{Artin}, or \cite[pp. 30--34]{Cec1}).

Since a M\"{o}bius transformation $\alpha = P(B)$ for $B \in O(n+1,1)$ maps lightlike points to lightlike
points in ${\bf P}^{n+1}$ in a bijective way, it induces a diffeomorphism of the $n$-sphere $\Sigma$ which is
conformal by the considerations given above.  It is well known that the group of conformal diffeomorphisms of
the $n$-sphere is precisely the M\"{o}bius group. 

\section{Lie Geometry of Oriented Spheres}
\label{sec:4}
We now turn to the construction of Lie's geometry of 
oriented spheres in ${\bf R}^n$.
Let $W^{n+1}$ be the set of vectors in ${\bf R}^{n+2}_1$ satisfying $(\zeta, \zeta) = 1.$  This
is a hyperboloid of revolution of one sheet in ${\bf R}^{n+2}_1$.  If $\alpha$ is a spacelike
point in ${\bf P}^{n+1}$, then there are precisely two vectors $\pm \zeta$ in $W^{n+1}$ with
$\alpha = [\zeta]$.  These two vectors can be taken to correspond to the two orientations
of the oriented sphere or plane represented by $\alpha$, as we will now describe.  We first introduce one
more coordinate.  We embed ${\bf R}^{n+2}_1$ into ${\bf P}^{n+2}$ by the embedding
$z \mapsto [(z,1)]$.  If $\zeta \in W^{n+1}$, then
\begin{displaymath}
- \zeta_1^2 + \zeta_2^2 + \cdots + \zeta_{n+2}^2 = 1,
\end{displaymath}
so the point $[(\zeta, 1)]$ in ${\bf P}^{n+2}$ lies on the quadric $Q^{n+1}$ in ${\bf P}^{n+2}$
given in homogeneous coordinates by the equation
\begin{equation}
\label{eq:1.3.1}
\langle x,x \rangle = - x_1^2 + x_2^2 + \cdots + x_{n+2}^2 - x_{n+3}^2 = 0.
\end{equation}
The manifold $Q^{n+1}$ is called the {\em Lie quadric}, and the scalar product
determined by the quadratic form in (\ref{eq:1.3.1}) is called the {\em Lie metric}
or {\em Lie scalar product}. 
We will let $\{e_1,\ldots,e_{n+3}\}$ denote the standard orthonormal
basis for the scalar product space ${\bf R}^{n+3}_2$ with metric $\langle \  ,  \ \rangle$.  Here $e_1$
and $e_{n+3}$ are timelike and the rest are spacelike.

We shall now see how points on $Q^{n+1}$ correspond to the set of oriented hyperspheres,
oriented hyperplanes
and point spheres in ${\bf R}^n \cup \{ \infty\}$.  Suppose that $x$ is any point on the 
quadric with homogeneous coordinate $x_{n+3} \neq 0$.  Then $x$ can be represented by a vector of the form
$(\zeta, 1)$, where the Lorentz scalar product $(\zeta, \zeta) = 1$.  Suppose first that
$\zeta_1 + \zeta_2 \neq 0$.  Then in M\"{o}bius geometry, $[\zeta]$ represents a sphere in ${\bf R}^n$.  If as in
equation (\ref{eq:1.2.4}), we represent $[\zeta]$ by a vector of the form
\begin{displaymath}
\xi = \left( \frac{1+p \cdot p - r^2}{2}, \frac{1- p \cdot p +r^2}{2},p \right),
\end{displaymath}
then $(\xi, \xi) = r^2$.  Thus $\zeta$ must be one of the vectors $\pm \xi/r$.  In ${\bf P}^{n+2}$, we have
\begin{displaymath}
[(\zeta, 1)] = [(\pm \xi / r, 1)] = [(\xi, \pm r)].
\end{displaymath}
We can interpret the last coordinate as a signed radius of the sphere with center $p$ and unsigned
radius $r > 0$.  In order to interpret this geometrically, we adopt the convention that
a positive signed radius 
corresponds to the orientation of the sphere
represented by the inward field
of unit normals, and a negative signed radius corresponds to the orientation given by the 
outward field of unit
normals.  Hence, the two orientations of the sphere in ${\bf R}^n$ with center $p$ and unsigned radius $r > 0$ are
represented by the two projective points,
\begin{equation}
\label{eq:1.3.2}
\left[ \left( \frac{1+p \cdot p - r^2}{2}, \frac{1- p \cdot p +r^2}{2},p, \pm r \right) \right]
\end{equation}
in $Q^{n+1}$.  Next if $\zeta_1 + \zeta_2 = 0$, then $[\zeta]$ represents a 
hyperplane in ${\bf R}^n$,
as in equation (\ref{eq:1.2.6}).  For $\zeta = (h,-h, N)$, with $|N| = 1$, we have $(\zeta, \zeta) = 1$.  Then the two
projective points on $Q^{n+1}$ induced by $\zeta$ and $- \zeta$ are
\begin{equation}
\label{eq:1.3.3}
[(h, -h, N, \pm 1)].
\end{equation}
These represent the two orientations of the plane with equation $u \cdot N = h$.  We make the
convention that $[(h, -h, N, 1)]$ corresponds to the orientation given by the field of unit
normals $N$, while the orientation given by $-N$ corresponds to the point
$[(h, -h, N, -1)] = [(-h, h, -N, 1)]$.

Thus far we have determined a bijective correspondence between the set of points $x$ in $Q^{n+1}$ with
$x_{n+3} \neq 0$ and the set of all oriented spheres
and planes
in ${\bf R}^n$.  Suppose now that
$x_{n+3} = 0$, i.e., consider a point $[(z,0)]$, for $z \in {\bf R}^{n+2}_1$.  Then
$\langle x, x \rangle = (z, z) = 0$, and $[z] \in {\bf P}^{n+1}$ is simply a point of the 
M\"{o}bius sphere $\Sigma$.  Thus we have the following bijective correspondence between objects
in Euclidean space and points on the Lie quadric: 

\begin{equation}
\label{eq:1.3.4}
\begin{array}{cc}
{\rm {\bf Euclidean}} & {\rm {\bf Lie}} \\
 & \\
\mbox{{\rm points:} }u \in {\bf R}^n & \left[ \left( \frac{1+u \cdot u}{2},\frac{1-u \cdot u}{2},u,0 \right) 
\right]  \\
 & \\
\infty & [(1,-1,0,0)]\\
 & \\
\mbox{{\rm spheres: center} $p$, {\rm signed radius} $r$} & \left[ \left( \frac{1+p \cdot p - r^2}{2}, 
\frac{1- p \cdot p +r^2} {2},p, r \right) \right] \\
 & \\ 
\mbox{{\rm planes:} $u \cdot N = h$, {\rm unit normal} $N$} & [(h,-h,N,1)] 
\end{array}
\end{equation}

\vspace{.25in}
\noindent
In Lie sphere geometry, points are considered to be spheres of radius zero, or ``point spheres.''  Point spheres do not
have an orientation.

From now on, we will use the term {\em Lie sphere} 
or simply ``sphere'' to denote an oriented
sphere, oriented plane or a point sphere in ${\bf R}^n \cup \{ \infty\}$.  We will refer to the
coordinates on the right side of equation (\ref{eq:1.3.4}) as the 
{\em Lie coordinates} of the corresponding
point, sphere or plane.  In the case of ${\bf R}^2$ and ${\bf R}^3$, respectively, these coordinates
were classically called {\em pentaspherical} and {\em hexaspherical}
coordinates (see Blaschke \cite{Bl}).  

At times
it is useful to have formulas to convert Lie coordinates back into Cartesian equations for the corresponding
Euclidean object.  Suppose first that $[x]$ is a point on the Lie quadric
with $x_1 + x_2 \neq 0$.
Then $x = \rho y$, for some $\rho \neq 0$, where $y$ is one of the standard forms on the right side of the table above.
From the table, we see that $y_1 + y_2 = 1$, for all proper points and all spheres.  Hence if we divide $x$
by $x_1 + x_2$, the new vector will be in standard form, and we can read off the corresponding Euclidean
object from the table.  In particular, if $x_{n+3} = 0$, then $[x]$ represents the point sphere $u = (u_3,\ldots,u_{n+2})$ where
\begin{equation}
\label{eq:1.3.5}
u_i = x_i / (x_1 + x_2),\quad 3 \leq i \leq n+2.
\end{equation}
If $x_{n+3} \neq 0$, then $[x]$ represents the sphere with center $p = (p_3,\ldots,p_{n+2})$ and signed
radius $r$ given by
\begin{equation}
\label{eq:1.3.6}
p_i = x_i / (x_1 + x_2),\quad 3 \leq i \leq n+2; \quad r = x_{n+3} / (x_1 + x_2).
\end{equation}
Finally, suppose that $x_1 + x_2 = 0$. If $x_{n+3} = 0$, then the equation $\langle x,x \rangle = 0$
forces $x_i$ to be zero for $3 \leq i \leq n+2$.  Thus $[x] = [(1,-1,0,\ldots,0)]$, the
improper point.  If $x_{n+3} \neq 0$, we divide $x$ by $x_{n+3}$ to make the last coordinate 1.
Then if we set $N = (N_3,\ldots,N_{n+2})$ and $h$ according to
\begin{equation}
\label{eq:1.3.7}
N_i = x_i/ x_{n+3},\quad 3 \leq i \leq n+2;\quad h = x_1/x_{n+3},
\end{equation}
the conditions $\langle x,x \rangle = 0$ and $x_1 + x_2 = 0$ force $N$ to have unit length.  Thus
$[x]$ corresponds to the hyperplane $u \cdot N = h$, with unit normal $N$ and $h$ as in equation (\ref{eq:1.3.7}).

If we wish to consider oriented hyperspheres
and point spheres in the unit sphere $S^n$ in ${\bf R}^{n+1}$, then the 
table (\ref{eq:1.3.4}) above can be simplified.  First, we have shown that in M\"{o}bius geometry, the unoriented 
hypersphere $S$ in $S^n$ with center $p \in S^n$ and spherical radius $\rho$, $0 < \rho < \pi$, corresponds to
the point $[\xi] = [(\cos \rho, p)]$ in ${\bf P}^{n+1}$.  To correspond the two orientations of this sphere to points on
the Lie quadric, we first note that
\begin{displaymath}
(\xi,\xi) = - \cos^2 \rho + 1 = \sin^2 \rho.
\end{displaymath}
Since $\sin \rho > 0$ for $0 < \rho < \pi$, we can divide $\xi$ by $\sin \rho$ and consider the two vectors
$\zeta = \pm \xi/\sin \rho$ that satisfy $(\zeta, \zeta) = 1$.  We then map these two points into the Lie quadric to
get the points
\begin{displaymath}
[(\zeta, 1)] = [(\xi, \pm \sin \rho)] = [(\cos \rho, p, \pm \sin \rho)].
\end{displaymath}
in $Q^{n+1}$.  We can incorporate the sign of the last coordinate into the radius and thereby arrange that the 
oriented sphere $S$ with 
signed radius $\rho \neq 0$, where $- \pi < \rho < \pi$, and center $p$ corresponds to the point 
\begin{equation}
\label{eq:1.4.4}
[x] = [(\cos \rho, p, \sin \rho)].
\end{equation}
in $Q^{n+1}$.  This formula still makes sense if the radius $\rho = 0$, in which case it yields the point sphere
$[(1, p, 0)]$.  

We adopt the convention that the positive radius $\rho$ in 
(\ref{eq:1.4.4}) corresponds to the orientation of the sphere given by 
the field of unit normals which are tangent vectors to geodesics in $S^n$ from $-p$ to $p$, and a negative radius corresponds
to the opposite orientation.
Each oriented sphere can be considered in two ways, with center $p$ and signed radius $\rho, - \pi < \rho < \pi$,
or with center $-p$ and the appropriate signed radius $\rho \pm \pi$.

For a given point $[x]$ in the quadric $Q^{n+1}$, we can determine the 
corresponding oriented hypersphere or point sphere in $S^n$ as follows.
Multiplying by $-1$, if necessary, we can arrange that that the first coordinate $x_1$ of $x$ is nonnegative.
If $x_1$ is positive, then it follows from equation (\ref{eq:1.4.4}) that the center $p$ and signed radius
$\rho, - \pi/2 < \rho < \pi/2$, are given by
\begin{equation}
\label{eq:1.4.5}
\tan \rho = x_{n+3}/x_1, \quad p = (x_2, \ldots, x_{n+2})/ (x_1^2 + x_{n+3}^2)^{1/2}.
\end{equation}
If $x_1 = 0$, then $x_{n+3}$ is nonzero, and we can divide by $x_{n+3}$ to obtain a point with coordinates
$(0,p,1)$.  This corresponds to the oriented hypersphere in $S^n$ with center $p$ and signed radius $\pi / 2$, which
is a great sphere in $S^n$.

\begin{remark}
{\rm In a similar way, one can develop the Lie sphere geometry of oriented spheres in real hyperbolic space $H^n$
(see, for example, \cite[p. 18]{Cec1}).}
\end{remark}

\section{Oriented Contact of Spheres}
\label{sec:5}

As we saw in Section \ref{sec:3}, the angle between two spheres is the fundamental geometric quantity in 
M\"{o}bius geometry, and it is the quantity that is preserved by M\"{o}bius transformations.  In Lie's geometry
of oriented spheres, the corresponding fundamental notion is that of oriented contact of spheres.
(See \cite[pp. 19--23]{Cec1} for more detail.)
  
By definition, two oriented spheres $S_1$ and $S_2$ in ${\bf R}^n$ are in {\em oriented contact}
if they are tangent to each other, and they have the same orientation at the point of contact
There are two geometric possibilities depending on whether the signed radii of $S_1$ and $S_2$ have the same sign or opposite signs. In either case,
if $p_1$ and $p_2$ are the respective centers of $S_1$ and $S_2$, and $r_1$ and $r_2$ are their respective
signed radii, then the analytic condition for oriented contact is
\begin{equation}
\label{eq:1.5.1}
|p_1 - p_2| = |r_1 - r_2|.
\end{equation}
Similarly, we say that an oriented hypersphere sphere $S$ with center $p$ and signed radius $r$ and an oriented
hyperplane $\pi$ with unit normal $N$ and equation $u \cdot N = h$ are in oriented contact
if $\pi$ is tangent to $S$ and their orientations agree at the point of contact.  This condition is given by the equation
\begin{equation}
\label{eq:1.5.2}
p \cdot N = r + h.
\end{equation}
Next we say that two oriented planes $\pi_1$ and $\pi_2$ are in oriented contact if their unit normals $N_1$ and $N_2$ are
the same.  These planes can be considered to be two oriented spheres in oriented contact at the improper point.
Finally, a proper point $u$ in ${\bf R}^n$ is in oriented contact with a sphere or a plane if it lies on the sphere or plane,
and the improper point is in oriented contact with each plane, since it lies on each plane.

An important fact in Lie sphere geometry is that if
$S_1$ and $S_2$ are two Lie spheres which are represented as in 
equation (\ref{eq:1.3.4}) by $[k_1]$ and $[k_2]$, then the analytic
condition for oriented contact is equivalent to the equation
\begin{equation}
\label{eq:1.5.3}
\langle k_1, k_2 \rangle = 0.
\end{equation}
This can be checked easily by a direct calculation.

By standard linear algebra in indefinite inner product spaces (see, for example, \cite[p. 21]{Cec1}), 
the fact that the signature of ${\bf R}^{n+3}_2$ is $(n+1,2)$ implies
that the Lie quadric contains projective lines in ${\bf P}^{n+2}$, but no linear subspaces of 
${\bf P}^{n+2}$ of
higher dimension. These projective lines on $Q^{n+1}$
play a crucial role in the theory of submanifolds in the
context of Lie sphere geometry.

One can show further (see \cite[pp. 21--23]{Cec1}), that if $[k_1]$ and $[k_2]$ are two points of 
$Q^{n+1}$, then  the line $[k_1,k_2]$ in ${\bf P}^{n+2}$ 
lies on $Q^{n+1}$ if and only if the spheres corresponding to $[k_1]$ and $[k_2]$ are in
oriented contact, i.e., $\langle k_1, k_2 \rangle = 0$.  Moreover,
if the line $[k_1,k_2]$ lies on $Q^{n+1}$, then the set of spheres in ${\bf R}^n$ corresponding
to points on the line $[k_1,k_2]$ is precisely the set of all spheres in oriented contact with both of these spheres. Such a 1-parameter family of spheres is called a {\em parabolic pencil} of spheres in 
${\bf R}^n \cup \{ \infty\}$.  

Each parabolic pencil contains exactly one point sphere, and if that point sphere is a proper point, then the parabolic pencil
contains exactly one hyperplane $\pi$ in ${\bf R}^n$, and
the pencil consists of all spheres in oriented contact with the oriented plane $\pi$ at $p$.  
Thus, we can associate the parabolic pencil
with the point $(p,N)$ in the unit tangent bundle 
of ${\bf R}^n \cup \{ \infty \}$, where $N$ is the unit normal to the oriented plane $\pi$.

If the point sphere in the pencil is the improper point, then the parabolic pencil
is a family of parallel hyperplanes in oriented contact at the improper point.  If $N$ is the common unit normal
to all of these planes, then we can associate the pencil with the point $(\infty, N)$ in the unit tangent
bundle of ${\bf R}^n \cup \{ \infty \}$.
  
Similarly, we can establish a correspondence between parabolic pencils and elements of the unit tangent bundle $T_1S^n$
that is expressed in terms of the spherical metric on $S^n$.  If $\ell$ is a line on the
quadric, then $\ell$ intersects both
$e_1^\perp$ and $e_{n+3}^\perp$ at exactly one point, where $e_1 = (1,0,\ldots,0)$ and $e_{n+3} = (0,\ldots,0,1)$.  
So the parabolic pencil corresponding to $\ell$ contains exactly
one point sphere (orthogonal to $e_{n+3}$) and one great sphere (orthogonal to $e_1$), 
given respectively by the points,
\begin{equation}
\label{eq:1.5.3a}
[k_1] = [(1,p,0)], \quad [k_2] = [(0,\xi,1)].
\end{equation}  
Since $\ell$ lies on the quadric, we know that 
$\langle k_1, k_2 \rangle = 0$, and this condition is equivalent to the condition $p \cdot \xi = 0$, i.e., $\xi$
is tangent to $S^n$ at $p$.  Thus,
the parabolic pencil of spheres corresponding to the line $\ell$ can
be associated with the point $(p, \xi)$ in $T_1S^n$.  More specifically, the line $\ell$ can be parametrized as
\begin{equation}
\label{eq:1.5.3b}
[K_t] = [\cos t \ k_1 + \sin t \ k_2] = [(\cos t, \cos t \ p + \sin t \ \xi, \sin t)].
\end{equation}
From equation (\ref{eq:1.4.4}) above, we see that $[K_t]$ corresponds to the oriented sphere in $S^n$ with center
\begin{equation}
\label{eq:1.5.4}
p_t = \cos t \ p + \sin t \ \xi,
\end{equation}  
and signed radius $t$.  The pencil consists of all oriented spheres in $S^n$
in oriented contact with the great sphere corresponding to $[k_2]$ at the point $(p, \xi)$ in $T_1S^n$.  
Their centers $p_t$ lie along the geodesic in $S^n$ with initial point $p$ and initial velocity
vector $\xi$. Detailed proofs of all these facts are given in \cite[pp. 21--23]{Cec1}.

We conclude this section with a discussion of Lie sphere transformations.  By definition,
a {\em Lie sphere transformation} is a projective transformation of ${\bf P}^{n+2}$ which maps
the Lie quadric $Q^{n+1}$ to itself.  In terms of the geometry of ${\bf R}^n$ or $S^n$, a Lie sphere transformation maps 
Lie spheres to Lie spheres, and since it is a projective transformation, 
it maps lines on $Q^{n+1}$ to lines on $Q^{n+1}$.  Thus, it preserves
oriented contact of spheres in ${\bf R}^n$ or $S^n$. Conversely, Pinkall \cite{P4}
(see also \cite[pp. 28--30]{Cec1}) proved the so-called 
``Fundamental Theorem of Lie sphere geometry,''
which states that any line preserving diffeomorphism 
of $Q^{n+1}$ is the restriction to $Q^{n+1}$
of a projective transformation, that is, a transformation of the space of oriented spheres which preserves oriented contact is a Lie sphere transformation.

By the same type of reasoning as given in Section \ref{sec:3} for M\"{o}bius transformations,
one can show that the group $G$ of Lie sphere transformations
is isomorphic to the group $O(n+1,2)/\{\pm I\}$, where $O(n+1,2)$ is the group of 
orthogonal transformations of ${\bf R}^{n+3}_2$.  As with the M\"{o}bius group, it follows from the
theorem of Cartan and Dieudonn\'{e} (see \cite[pp. 30--34]{Cec1}) that the Lie sphere group $G$ is generated by
Lie inversions, that is, projective transformations that are induced by reflections in $O(n+1,2)$.

The M\"{o}bius group $H$ can be considered to be a subgroup of $G$ in the following manner.
Each  M\"{o}bius transformation on the space of unoriented spheres, naturally induces two
Lie sphere transformations on the space $Q^{n+1}$ of oriented spheres as follows.  If $A$ is in $O(n+1,1)$, then
we can extend $A$ to a transformation $B$ in $O(n+1,2)$ by setting $B = A$ on ${\bf R}^{n+2}_1$ and
$B(e_{n+3}) = e_{n+3}$.  In terms the standard orthonormal basis in ${\bf R}^{n+3}_2$, the transformation
$B$ has the matrix representation,
\begin{equation}
\label{eq:2.1.4}
B = \left[ \begin{array}{cc}A&0\\0&1\end{array}\right].
\end{equation}
Although $A$ and $-A$ induce the same M\"{o}bius transformation in $H$, the Lie transformation $P(B)$
is not the same as the Lie transformation $P(C)$ induced by the matrix
\begin{displaymath}
C = \left[ \begin{array}{cc}-A&0\\0&1\end{array}\right] \simeq \left[ \begin{array}{cc}A&0\\0&-1\end{array}\right],
\end{displaymath}
where $\simeq$ denotes equivalence as projective transformations.  Note that $P(B) = \Gamma P(C)$, where $\Gamma$
is the Lie transformation represented in matrix form by
\begin{displaymath}
\Gamma = \left[ \begin{array}{cc}I&0\\0&-1\end{array}\right] \simeq \left[ \begin{array}{cc}-I&0\\0&1\end{array}\right].
\end{displaymath}
From equation (\ref{eq:1.3.4}), 
we see that $\Gamma$ has the effect of changing the orientation of
every oriented sphere or plane.  The transformation $\Gamma$ is called the 
{\em change of orientation transformation} or 
``Richtungswechsel'' in German. Hence, the two Lie sphere transformations induced by the
M\"{o}bius transformation $P(A)$ differ by this change of orientation factor.  

Thus, the group of Lie sphere transformations
induced from M\"{o}bius transformations is isomorphic to $O(n+1,1)$. 
This group consists of those Lie transformations that map $[e_{n+3}]$ to itself, and it is
a double covering of the M\"{o}bius group $H$.
Since these transformations are induced from orthogonal transformations of ${\bf R}^{n+3}_2$, they 
also map $e_{n+3}^\perp$ to itself, and thereby map point
spheres to point spheres.  When working in the context of Lie sphere geometry, we will refer to these
transformations as ``M\"{o}bius transformations.''

\section{Legendre Submanifolds}
\label{sec:6}
The goal of this section is to define a contact structure on the unit tangent bundle
$T_1S^n$ and on the $(2n-1)$-dimensional manifold
$\Lambda^{2n-1}$ of projective lines on the Lie quadric $Q^{n+1}$, and to describe its associated Legendre submanifolds.
This will enable us to study submanifolds of ${\bf R}^n$ or $S^n$ within the context of Lie sphere geometry in
a natural way.  This theory was first developed extensively in a modern
setting by Pinkall \cite{P4} (see also Cecil-Chern \cite{CC1}--\cite{CC2} or the books \cite[pp. 51--60]{Cec1},  \cite[pp. 202--212]{CR8}).

We consider $T_1S^n$ to be the $(2n-1)$-dimensional submanifold of 
\begin{displaymath}
S^n \times S^n \subset{\bf R}^{n+1} \times {\bf R}^{n+1}
\end{displaymath}
given by
\begin{equation}
\label{eq:3.1.1}
T_1S^n = \{(x, \xi) \mid \  |x| =1, \  |\xi| = 1, \  x \cdot \xi = 0\}.
\end{equation}

As shown in the previous section, the points on a line $\ell$ lying on $Q^{n+1}$ correspond to the spheres
in a parabolic pencil of spheres in $S^n$.  In particular, as in equation (\ref{eq:1.5.3a}), $\ell$ contains
one point $[k_1] = [(1,x,0)]$ corresponding to a point sphere in $S^n$, and one point
$[k_2] = [(0,\xi,1)]$ corresponding to a great sphere in $S^n$, where the coordinates are with respect to
the standard orthonormal basis $\{e_1,\ldots,e_{n+3}\}$ of ${\bf R}^{n+3}_2$. Thus we get a bijective
correspondence between the points $(x,\xi)$ of $T_1S^n$ and the 
space $\Lambda^{2n-1}$ of lines on $Q^{n+1}$
given by the map:
\begin{equation}
\label{eq:3.1.8}
(x, \xi) \mapsto [Y_1(x, \xi), Y_{n+3}(x, \xi)],
\end{equation}
where
\begin{equation}
\label{eq:3.1.9}
Y_1(x, \xi) = (1,x,0), \quad Y_{n+3}(x, \xi) = (0,\xi,1).
\end{equation}
We use this correspondence to place a natural differentiable structure on $\Lambda^{2n-1}$ in such a way as to
make the map in equation (\ref{eq:3.1.8}) a diffeomorphism.

We now show how to define a contact structure on the manifold $T_1S^n$.  By the diffeomorphism in
equation (\ref{eq:3.1.8}), this also determines a contact structure on $\Lambda^{2n-1}$.
Recall that a $(2n-1)$-dimensional manifold $V^{2n-1}$ is said to be a {\em contact manifold} 
if it carries a globally defined
1-form $\omega$ such that
\begin{equation}
\label{eq:3.1.2}
\omega \wedge (d\omega)^{n-1} \neq 0
\end{equation}
at all points of $V^{2n-1}$.  Such a form $\omega$ is called a 
{\em contact form}.  
A contact form $\omega$ determines a codimension one distribution (the {\em contact distribution})
$D$ on $V^{2n-1}$ defined by
\begin{equation}
\label{eq:3.1.3}
D_p = \{ Y \in T_pV^{2n-1} \mid \omega (Y) = 0\},
\end{equation}
for $p \in V^{2n-1}$.   
This distribution is as far from being integrable
as possible, in that there exist integral submanifolds of $D$ of dimension $n-1$, but none of higher dimension
(see, for example, \cite[p. 57]{Cec1}).  The distribution $D$ determines the corresponding contact form $\omega$ up to multiplication by a nonvanishing smooth function.

A tangent vector to $T_1S^n$ at a point $(x, \xi)$ can be written in the form $(X,Z)$ where
\begin{equation}
\label{eq:3.1.4}
X \cdot x = 0, \quad Z \cdot \xi = 0.
\end{equation}
Differentiation of the condition $x \cdot \xi = 0$ implies that $(X,Z)$ also satisfies
\begin{equation}
\label{eq:3.1.5}
X \cdot \xi + Z \cdot x = 0.
\end{equation}
Using the method of moving frames, 
one can show that the form $\omega$ defined by
\begin{equation}
\label{eq:3.1.6}
\omega (X,Z) = X \cdot \xi,
\end{equation}
is a contact form on $T_1S^n$  (see, for example, Cecil--Chern \cite{CC1} or the book \cite[pp. 52--56]{Cec1}),
and we will omit the proof here.

At a point $(x, \xi)$, the distribution $D$ is the $(2n-2)$-dimensional
space of vectors $(X,Z)$ satisfying $X \cdot \xi = 0$, as well as the equations (\ref{eq:3.1.4}) and (\ref{eq:3.1.5}).
The equation $X \cdot \xi = 0$ together with equation (\ref{eq:3.1.5}) implies that
\begin{equation}
\label{eq:3.1.7}
Z \cdot x = 0,
\end{equation}
for vectors $(X,Z)$ in $D$.

Returning to the general theory of contact structures, we let $V^{2n-1}$ be a contact manifold with contact form 
$\omega$ and corresponding 
contact distribution $D$, as in equation (\ref{eq:3.1.3}).
An immersion $\phi: W^k \rightarrow V^{2n-1}$ of a smooth $k$-dimensional manifold
$W^k$ into $V^{2n-1}$ is called an {\em integral submanifold}
of the distribution $D$ if 
$\phi^*\omega = 0$ on $W^k$, i.e., for each tangent vector $Y$ at each point $w \in W$, the vector
$d\phi (Y)$ is in the distribution $D$ at the point $\phi(w)$. (See Blair \cite[p. 36]{Blair}.)
It is well known (see, for example, \cite[p. 57]{Cec1}) that the contact distribution $D$ has integral
submanifolds of dimension $n-1$, but none of higher dimension.  These integral submanifolds of
maximal dimension are called {\em Legendre submanifolds} of the contact structure.

In our specific case, we now formulate conditions for a
smooth map $\mu: M^{n-1} \rightarrow T_1S^n$ to be a Legendre submanifold.  
We consider $T_1S^n$ as a submanifold of $S^n \times S^n$ as in equation (\ref{eq:3.1.1}), and so
we can write $\mu = (f,\xi)$, where $f$ and $\xi$ are both smooth maps from $M^{n-1}$ to $S^n$.
We have the following theorem (see \cite[p. 58]{Cec1}) giving necessary and sufficient conditions for $\mu$ to be
a Legendre submanifold.

\begin{theorem}
\label{thm:3.2.2} 
A smooth map $\mu = (f,\xi)$ from an $(n-1)$-dimensional manifold $M^{n-1}$ into $T_1S^n$ is a Legendre submanifold
if and only if the following three conditions are satisfied.
\begin{enumerate}
\item[$(1)$]  Scalar product conditions: $f \cdot f =1, \quad \xi \cdot \xi = 1, \quad f \cdot \xi = 0$.
\item[$(2)$] Immersion condition: there is no nonzero tangent vector $X$ at any point $x \in M^{n-1}$ such that $df(X)$ and $d\xi(X)$ are both equal to zero.
\item[$(3)$] Contact condition: $df \cdot \xi = 0$.
\end{enumerate}
\end{theorem}

Note that by equation (\ref{eq:3.1.1}), the scalar product conditions are precisely the conditions 
necessary for the image of the map $\mu = (f,\xi)$ to be contained in $T_1S^n$.  Next, since
$d\mu (X) = (df(X), d\xi (X))$, Condition $(2)$ is necessary and sufficient for $\mu$ to be an immersion.  Finally, from equation (\ref{eq:3.1.6}), we see that $\omega (d\mu (X)) = df(X) \cdot \xi (x)$,
for each $X \in T_xM^{n-1}$. Hence Condition $(3)$ is equivalent to the requirement that
$\mu^*\omega = 0$ on $M^{n-1}$.

We now want to translate these conditions into the projective setting, and find necessary and sufficient
conditions for a smooth map $\lambda:M^{n-1} \rightarrow \Lambda^{2n-1}$ to be a Legendre submanifold.  We again
make use of the diffeomorphism defined in equation (\ref{eq:3.1.8}) between $T_1S^n$ and $\Lambda^{2n-1}$.  

For each $x \in M^{n-1}$, we know that $\lambda(x)$ is a line on the 
quadric $Q^{n+1}$.  This line contains exactly one
point $[Y_1(x)] = [(1, f(x), 0)]$ corresponding to a point sphere in $S^n$,
and one point $[Y_{n+3}(x)] = [(0,\xi (x), 1)]$ corresponding to a
great sphere in $S^n$. These two formulas define maps $f$ and $\xi$ from $M^{n-1}$ to $S^n$ which
depend on the choice of orthonormal basis $\{e_1,\ldots,e_{n+2}\}$
for the orthogonal complement of $e_{n+3}$.

The map $[Y_1]$ from $M^{n-1}$ to $Q^{n+1}$ is called the 
{\em M\"{o}bius projection} or {\em point sphere map}
of $\lambda$, and the map $[Y_{n+3}]$ from $M^{n-1}$ to $Q^{n+1}$ is called the 
{\em great sphere map}. The maps $f$ and $\xi$ are called the {\em spherical projection} of $\lambda$,
and the {\em spherical field of unit normals} of $\lambda$, respectively.

In this way, $\lambda$ determines a map $\mu = (f, \xi)$ from
$M^{n-1}$ to $T_1S^n$, and because of the diffeomorphism (\ref{eq:3.1.8}), $\lambda$ is a Legendre submanifold
if and only if $\mu$ satisfies the conditions of Theorem~\ref{thm:3.2.2}. 

It is often useful to have conditions for
when $\lambda$ determines a Legendre submanifold that do not depend on the special parametrization of 
$\lambda$ in terms of the point sphere and great sphere maps, $[Y_1]$ and $[Y_{n+3}]$.  
In fact, in many applications of Lie sphere geometry to submanifolds of $S^n$ or ${\bf R}^n$,
it is better to consider $\lambda = [Z_1, Z_{n+3}]$, where $Z_1$ and
$Z_{n+3}$ are not the point sphere and great sphere maps.

Pinkall \cite{P4} gave the following projective formulation of the conditions 
needed for a Legendre submanifold.  In his paper, Pinkall referred to a Legendre
submanifold as a ``Lie geometric hypersurface.''
The proof that the three conditions
of the theorem below
are equivalent to the three conditions of Theorem~\ref{thm:3.2.2} can be found in \cite[pp. 59--60]{Cec1}.

\begin{theorem}
\label{thm:3.2.3} 
Let $\lambda:M^{n-1} \rightarrow \Lambda^{2n-1}$ be a smooth map with $\lambda = [Z_1, Z_{n+3}]$, where $Z_1$ and
$Z_{n+3}$ are smooth maps from $M^{n-1}$ into ${\bf R}^{n+3}_2$.  Then $\lambda$ determines 
a Legendre submanifold if and only
if $Z_1$ and $Z_{n+3}$ satisfy the following conditions.
\begin{enumerate}
\item[$(1)$] Scalar product conditions:
for each $x \in M^{n-1}$, the vectors $Z_1(x)$ and $Z_{n+3}(x)$ are linearly independent
and 
\begin{displaymath}
\langle Z_1,Z_1 \rangle = 0, \quad \langle Z_{n+3}, Z_{n+3} \rangle = 0, \quad \langle Z_1, Z_{n+3} \rangle = 0.
\end{displaymath} 
\item[$(2)$] Immersion condition:
there is no nonzero tangent vector $X$ at any point $x \in M^{n-1}$ such that
$dZ_1(X)$ and $dZ_{n+3}(X)$ are both in 
\begin{displaymath}
{\rm Span}\  \{Z_1(x),Z_{n+3}(x)\}.
\end{displaymath}
\item[$(3)$] Contact condition: $\langle dZ_1, Z_{n+3} \rangle = 0$.
\end{enumerate}
These conditions are invariant under a reparametrization $\lambda = [W_1, W_{n+3}]$, where 
$W_1 = \alpha Z_1 + \beta Z_{n+3}$ and $W_{n+3} = \gamma Z_1 + \delta Z_{n+3}$, for smooth functions
$\alpha, \beta, \gamma, \delta$ on $M^{n-1}$ with $\alpha \delta - \beta \gamma \neq 0.$ 
\end{theorem}

Every oriented hypersurface in $S^n$ or ${\bf R}^n$ naturally
induces a Legendre submanifold of $\Lambda^{2n-1}$, as does every submanifold of codimension $m>1$ in these spaces.
Conversely, a Legendre submanifold naturally induces a
smooth map into $S^n$ or ${\bf R}^n$, which may have singularities.  We now study the details of these maps.

Let $f:M^{n-1} \rightarrow S^n$ be an immersed oriented hypersurface with field of unit normals
$\xi:M^{n-1} \rightarrow S^n$.  The induced Legendre submanifold 
is given by the map $\lambda:M^{n-1} \rightarrow \Lambda^{2n-1}$ defined by 
$\lambda(x) = [Y_1(x), Y_{n+3} (x)]$, where
\begin{equation}
\label{eq:3.3.1}
Y_1(x) = (1, f(x), 0), \quad Y_{n+3}(x) = (0,\xi (x), 1).
\end{equation} 
The map $\lambda$ is called the {\em Legendre lift}
of the immersion $f$ with field of unit normals
$\xi$. 

To show that $\lambda$ is a Legendre submanifold, we check
the conditions of Theorem~\ref{thm:3.2.3}. Condition (1) is satisfied since both
$f$ and $\xi$ are maps into $S^n$, and $\xi(x)$ is tangent to $S^n$ at $f(x)$ for each $x$ in $M^{n-1}$.
Since $f$ is an immersion, $dY_1(X)= (0,df(X),0)$ is not in Span $\{Y_1(x),Y_{n+3}(x)\}$,
for any nonzero vector $X \in T_xM^{n-1}$, and so Condition $(2)$ is satisfied. Finally, Condition (3) is satisfied since
\begin{displaymath}
\langle dY_1(X), Y_{n+3}(x) \rangle = df(X) \cdot \xi(x) = 0,
\end{displaymath} 
because $\xi$ is a field of unit normals to $f$.

In the case of a submanifold $\phi: V \rightarrow S^n$ of codimension $m+1$ greater than one, the domain of the
Legendre lift is
be the unit normal bundle $B^{n-1}$ 
of the submanifold $\phi (V)$.  We consider $B^{n-1}$ to be the
submanifold of $V \times S^n$ given by
\begin{displaymath}
B^{n-1} = \{ (x,\xi)| \phi(x) \cdot \xi = 0, \ d\phi(X) \cdot \xi = 0,\  \mbox{\rm for all }X \in T_xV\}.
\end{displaymath} 
The {\em Legendre lift}  of
$\phi$  is the map $\lambda: B^{n-1} \rightarrow \Lambda^{2n-1}$ defined by
\begin{equation}
\label{eq:3.3.2}
\lambda(x,\xi) = [Y_1(x,\xi), Y_{n+3}(x,\xi)],
\end{equation} 
where
\begin{equation}
\label{eq:3.3.3}
Y_1(x,\xi) = (1, \phi(x), 0), \quad Y_{n+3}(x,\xi) = (0, \xi, 1).
\end{equation} 
Geometrically, $\lambda(x,\xi)$ is the line on the quadric $Q^{n+1}$ corresponding to the 
parabolic parabolic pencil
of spheres in $S^n$ in oriented contact at the contact element $(\phi(x),\xi) \in T_1S^n$.
In \cite[pp. 61--62]{Cec1}, we show that $\lambda$ satisfies the conditions of Theorem~\ref{thm:3.2.3},
and we omit the proof here.

Similarly, suppose that
$F:M^{n-1} \rightarrow {\bf R}^n$ is an oriented hypersurface with field of unit normals $\eta:M^{n-1} \rightarrow 
{\bf R}^n$, where we identify ${\bf R}^n$ with the subspace of ${\bf R}^{n+3}_2$ spanned by
$\{e_3, \ldots, e_{n+2}\}$.  The Legendre lift of $(F, \eta)$ is the map 
$\lambda:M^{n-1} \rightarrow \Lambda^{2n-1}$ defined by $\lambda = [Y_1, Y_{n+3}]$, where
\begin{equation}
\label{eq:3.3.7}
Y_1 = (1+F\cdot F, 1 - F \cdot F, 2F,0)/2, \quad Y_{n+3} = (F \cdot \eta, - (F \cdot \eta), \eta, 1).
\end{equation} 
By equation (\ref{eq:1.3.4}), $[Y_1(x)]$ corresponds to the point sphere and
$[Y_{n+3}(x)]$ corresponds to the hyperplane
in the parabolic pencil determined by the line $\lambda(x)$ for each
$x \in M^{n-1}$.  One can easily verify that Conditions (1)--(3) of Theorem~\ref{thm:3.2.3} 
are satisfied in a manner similar
to the spherical case.  In the case of a submanifold $\psi:V \rightarrow {\bf R}^n$ of codimension greater
than one, the Legendre lift of $\psi$ is the map $\lambda$ from the unit normal
bundle $B^{n-1}$ to $\Lambda^{2n-1}$ defined by 
$\lambda(x,\eta) = [Y_1(x,\eta), Y_{n+3}(x,\eta)]$, where
\begin{eqnarray}
\label{eq:3.3.8}
Y_1(x, \eta) & = & (1 + \psi(x) \cdot \psi(x),1 - \psi(x) \cdot \psi(x), 2 \psi(x), 0)/2,\\
Y_{n+3}(x, \eta) & = & (\psi(x) \cdot \eta, -(\psi(x) \cdot \eta), \eta, 1). \nonumber
\end{eqnarray}
The verification that the pair $\{Y_1, Y_{n+3}\}$ satisfies conditions (1)--(3) of Theorem~\ref{thm:3.2.3}
is similar to that for submanifolds of
$S^n$ of codimension greater than one,
and we omit that proof here also.

Conversely, suppose that $\lambda: M^{n-1} \rightarrow \Lambda^{2n-1}$ is an arbitrary Legendre submanifold. 
We have seen above that we can parametrize $\lambda$ as $\lambda = [Y_1, Y_{n+3}]$, where 
\begin{equation}
\label{eq:3.3.10}
Y_1 = (1, f, 0), \quad Y_{n+3} = (0,\xi, 1),
\end{equation} 
for the spherical projection $f$ and spherical
field of unit normals $\xi$. Both $f$ and $\xi$ are smooth maps, but neither need be
an immersion or even have constant rank (see \cite[pp. 63--64]{Cec1}).

The Legendre lift of an oriented
hypersurface in $S^n$ is the special case where the spherical projection 
$f$ is an immersion, i.e., $f$ has constant
rank $n-1$ on $M^{n-1}$.  In the case of the Legendre lift of a submanifold 
$\phi:V^k \rightarrow S^n$, the spherical
projection $f:B^{n-1} \rightarrow S^n$ defined by $f(x,\xi) = \phi (x)$ has constant 
rank $k$.

If the range of the point sphere map $[Y_1]$ does not contain the improper point $[(1,-1,0,\ldots,0)]$, then
$\lambda$ also determines a {\em Euclidean projection}  $F$, where
$F:M^{n-1} \rightarrow {\bf R}^n$,
and a {\em Euclidean field of unit normals} $\eta$, where
$\eta:M^{n-1} \rightarrow {\bf R}^n$.
These are defined by the equation 
$\lambda = [Z_1, Z_{n+3}]$, where
\begin{equation}
\label{eq:3.3.11}
Z_1 = (1+F\cdot F, 1 - F \cdot F, 2F,0)/2, \quad Z_{n+3} = (F \cdot \eta, - (F \cdot \eta), \eta, 1).
\end{equation} 
Here $[Z_1(x)]$ corresponds to the unique point sphere in the parabolic pencil determined by $\lambda(x)$,
and $[Z_{n+3}(x)]$ corresponds to the unique plane in this pencil.  As in the spherical case, the smooth maps
$F$ and $\eta$ need not have constant rank.

\section{Curvature Spheres}
\label{sec:7}
To motivate the definition of a curvature sphere
we consider the case of an oriented hypersurface
$f:M^{n-1} \rightarrow S^n$ with field of unit normals $\xi:M^{n-1} \rightarrow S^n$.  
(We could consider an oriented hypersurface in ${\bf R}^n$, but the calculations are simpler in the spherical case.)

The shape operator
of $f$ at a point $x \in M^{n-1}$ is the symmetric linear transformation $A:T_xM^{n-1} \rightarrow T_xM^{n-1}$
defined on the tangent space $T_xM^{n-1}$ by the equation
\begin{equation}
\label{eq:3.4.1}
df(AX) = - d\xi(X), \quad X \in T_xM^{n-1}.
\end{equation} 
The eigenvalues of $A$ are called the 
{\em principal curvatures}, and the corresponding eigenvectors are
called the {\em principal vectors}.  
We next recall the notion of a focal point of an immersion.  For each
real number $t$, define a map
\begin{displaymath}
f_t:M^{n-1} \rightarrow S^n,
\end{displaymath}
by
\begin{equation}
\label{eq:3.4.2}
f_t = \cos t \ f + \sin t \ \xi.
\end{equation} 
For each $x \in M^{n-1}$, the point $f_t(x)$ lies an oriented distance $t$ along the normal geodesic
to $f(M^{n-1})$ at $f(x)$.  A point $p = f_t(x)$ is called a 
{\em focal point of multiplicity} $m>0$
{\em of} $f$ {\em at} $x$ if the nullity of $df_t$ is equal to $m$ at $x$.  Geometrically, one thinks of focal points
as points where nearby normal geodesics intersect.  It is well known that the location of focal points is related to the principal curvatures.  Specifically, if $X \in T_xM^{n-1}$, then by equation (\ref{eq:3.4.1}) we have
\begin{equation}
\label{eq:3.4.3}
df_t(X) = \cos t \ df(X) + \sin t \ d\xi(X) = df(\cos t \ X - \sin t \ AX).
\end{equation} 
Thus, $df_t(X)$ equals zero for $X\neq0$ if and only if $\cot t$ is a principal curvature of $f$ at $x$,
and $X$ is a corresponding principal vector.  Hence, $p = f_t(x)$ is a focal point of $f$ at $x$ of
multiplicity $m$ if and only if $\cot t$ is a principal curvature of multiplicity $m$ at $x$.  Note
that each principal curvature 
\begin{displaymath}
\kappa = \cot t, \quad 0<t<\pi,
\end{displaymath}
produces two distinct antipodal focal points
on the normal geodesic with parameter values $t$ and $t+\pi$.  The oriented hypersphere centered at a focal point
$p$ and in oriented contact with $f(M^{n-1})$ at $f(x)$ is called a 
{\em curvature sphere} of $f$ at $x$.  The
two antipodal focal points determined by $\kappa$ are the two centers of the corresponding curvature sphere.
Thus, the correspondence between principal curvatures and curvature spheres is bijective.  The multiplicity
of the curvature sphere is by definition equal to the multiplicity of the corresponding principal curvature.

We now formulate the notion of curvature sphere in the context of Lie sphere geometry.  
As in equation (\ref{eq:3.3.1}), the Legendre lift $\lambda: M^{n-1} \rightarrow \Lambda^{2n-1}$ 
of the oriented hypersurface $(f,\xi)$ is given by $\lambda = [Y_1, Y_{n+3}]$, where
\begin{equation}
\label{eq:3.4.4}
Y_1 = (1, f, 0), \quad Y_{n+3} = (0, \xi, 1).
\end{equation} 
For each $x \in M^{n-1}$, the points on the line $\lambda(x)$ can be parametrized as
\begin{equation}
\label{eq:3.4.5}
[K_t(x)] = [\cos t \ Y_1(x) + \sin t \ Y_{n+3}(x)] = [(\cos t,f_t(x),\sin t)],
\end{equation} 
where $f_t$ is given in equation (\ref{eq:3.4.2}) above.  By equation (\ref{eq:1.4.4}), the point
$[K_t(x)]$ in $Q^{n+1}$ corresponds to the oriented sphere in $S^n$ with center $f_t(x)$ and signed 
radius $t$.
This sphere is in oriented contact with the oriented hypersurface $f(M^{n-1})$ at $f(x)$.  Given a tangent vector
$X \in T_xM^{n-1}$, we have
\begin{equation}
\label{eq:3.4.6}
dK_t(X)= (0,df_t(X),0).
\end{equation} 
Thus, $dK_t(X)= (0,0,0)$ for a nonzero vector $X \in T_xM^{n-1}$ if and only if $df_t(X) = 0$, 
i.e., $p = f_t(x)$ is a focal point of $f$ at $x$
corresponding to the principal curvature $\cot t$.  The vector $X$ is a principal vector corresponding to the principal 
curvature $\cot t$, and it is also called a principal vector corresponding to the curvature sphere $[K_t]$.

This characterization of curvature spheres depends on the parametrization of $\lambda = [Y_1, Y_{n+3}]$ given by
the point sphere and great sphere maps $[Y_1]$ and $[Y_{n+3}]$, respectively,
and it has only been defined in the case where the spherical projection $f$ is an immersion.
We now give a projective formulation of the definition of a curvature sphere that is
independent of the parametrization of $\lambda$ and is valid for 
an arbitrary Legendre submanifold. 

Let $\lambda: M^{n-1} \rightarrow \Lambda^{2n-1}$ be a Legendre submanifold parametrized by the pair 
$\{Z_1, Z_{n+3}\}$, as in Theorem~\ref{thm:3.2.3}.  Let $x \in M^{n-1}$ and $r,s \in {\bf R}$ with
at least one of $r$ and $s$ not equal to zero.  The sphere,
\begin{displaymath}
[K] = [r Z_1(x) + s Z_{n+3}(x)],
\end{displaymath}
is called a {\em curvature sphere}
of $\lambda$ at $x$ if there exists a nonzero vector $X$ in $T_xM^{n-1}$ such that
\begin{equation}
\label{eq:3.4.7}
r \ dZ_1(X) + s \ dZ_{n+3}(X) \in \mbox{\rm Span } \{Z_1(x),Z_{n+3}(x)\}.
\end{equation} 
The vector $X$ is called a {\em principal vector} 
corresponding to the curvature sphere $[K]$.  
This definition is invariant under a change of parametrization of the form considered in the statement of 
Theorem~\ref{thm:3.2.3}.  Furthermore, if we take the special parametrization $Z_1 = Y_1$, $Z_{n+3} = Y_{n+3}$
given in equation (\ref{eq:3.4.4}), then condition (\ref{eq:3.4.7}) holds if and only if $r \ dY_1(X) + s \ dY_{n+3}(X)$
actually equals $(0,0,0)$. 

From equation (\ref{eq:3.4.7}), it is clear that the set of principal vectors corresponding to a given curvature
sphere $[K]$ at $x$ is a subspace of $T_xM^{n-1}$.  This set is called the 
{\em principal space} corresponding
to the curvature sphere $[K]$.  Its dimension is the 
{\em multiplicity} of $[K]$. 
The reader is referred to Cecil--Chern
\cite{CC1}--\cite{CC2} for a development of the notion of a curvature sphere in the context of Lie sphere geometry, without beginning with submanifolds of $S^n$ or ${\bf R}^n$.

We next show that a Lie sphere transformation maps curvature spheres to curvature spheres.  We first need
to discuss the notion of Lie equivalent Legendre submanifolds.
Let $\lambda: M^{n-1} \rightarrow \Lambda^{2n-1}$ be a Legendre submanifold parametrized by $\lambda = [Z_1, Z_{n+3}]$.
Suppose $\beta = P(B)$ is the Lie sphere transformation induced by an 
orthogonal transformation $B$ in the group $O(n+1,2)$.
Since $B$ is orthogonal, the maps, $W_1 = BZ_1$, $W_{n+3} = BZ_{n+3}$, satisfy the 
Conditions (1)--(3) of Theorem~\ref{thm:3.2.3}, and thus $\gamma = [W_1, W_{n+3}]$ is a Legendre submanifold
which we denote by $\beta \lambda: M^{n-1} \rightarrow \Lambda^{2n-1}$.
We say that the Legendre submanifolds $\lambda$ 
and $\beta \lambda$ are {\em Lie equivalent}.  In terms of submanifolds of real space forms,
we say that two immersed submanifolds of ${\bf R}^n$ or $S^n$  are 
{\em Lie equivalent} if their Legendre lifts are Lie equivalent.

\begin{theorem}
\label{thm:3.4.3} 
Let $\lambda: M^{n-1} \rightarrow \Lambda^{2n-1}$ be a Legendre submanifold and $\beta$ a Lie sphere transformation.
The point $[K]$ on the line $\lambda(x)$ is a curvature sphere of $\lambda$ at $x$ if and only if the
point $\beta [K]$ is a curvature sphere of the Legendre submanifold $\beta \lambda$ at $x$.  Furthermore,
the principal spaces corresponding to $[K]$ and $\beta [K]$ are identical. 
\end{theorem}

\begin{proof}
Let $\lambda = [Z_1, Z_{n+3}]$ and $\beta \lambda = [W_1, W_{n+3}]$ as above.
For a tangent vector $X \in T_xM^{n-1}$ and real numbers $r$ and $s$, at least one of which is not zero, we have
\begin{eqnarray}
\label{eq:3.4.8}
r \ dW_1(X) + s \ dW_{n+3}(X) & = & r\ d(BZ_1)(X) + s\ d(BZ_{n+3})(X)\\
& = &B (r \ dZ_1(X) + s \ dZ_{n+3}(X)),\nonumber
\end{eqnarray} 
since $B$ is a constant linear transformation.  Thus, we see that
\begin{displaymath}
r \ dW_1(X) + s \ dW_{n+3}(X) \in \mbox{\rm Span } \{W_1(x),W_{n+3}(x)\}
\end{displaymath} 
if and only if
\begin{displaymath}
r \ dZ_1(X) + s \ dZ_{n+3}(X) \in \mbox{\rm Span } \{Z_1(x),Z_{n+3}(x)\}.
\end{displaymath} 
\end{proof}

We next consider the case when the Lie sphere transformation $\beta$ is a 
spherical parallel transformation $P_t$ defined in terms of the standard basis of ${\bf R}^{n+3}_2$ by
\begin{eqnarray}
\label{eq:3.4.9}
P_t e_1 & = & \cos t \ e_1 + \sin t \ e_{n+3}, \nonumber \\
P_t e_{n+3} & = & - \sin t \ e_1 + \cos t \ e_{n+3},\\
P_t e_i & = & e_i, \quad 2 \leq i \leq n+2. \nonumber
\end{eqnarray}
The transformation $P_t$ has the effect of adding $t$ to the signed radius 
of each oriented sphere in $S^n$ while keeping the center
fixed (see, for example, \cite[pp. 48--49]{Cec1}).

If $\lambda: M^{n-1} \rightarrow \Lambda^{2n-1}$ is a Legendre submanifold parametrized by the 
point sphere map $Y_1 = (1,f,0)$ and the great sphere map $Y_{n+3} = (0,\xi,1)$,  
then $P_t \lambda = [W_1, W_{n+3}]$, where 
\begin{equation}
\label{eq:3.4.10}
W_1 = P_t Y_1 = (\cos t, f,\sin t), \quad W_{n+3} = P_t Y_{n+3} = (- \sin t, \xi, \cos t).
\end{equation} 
Note that $W_1$ and $W_{n+3}$ are not the point sphere and great sphere maps for $P_t \lambda$.  Solving for the point sphere map $Z_1$ and the great sphere map $Z_{n+3}$ of $P_t \lambda$, we find
\begin{eqnarray}
\label{eq:3.4.11}
Z_1 & = & \cos t \ W_1 - \sin t \ W_{n+3} = (1, \cos t \ f - \sin t \ \xi, 0),  \\
Z_{n+3} & = &  \sin t \ W_1 + \cos t \ W_{n+3} = (0, \sin t \ f + \cos t \ \xi, 1).\nonumber
\end{eqnarray}
From this, we see that $P_t \lambda$ has spherical projection and spherical unit normal field given, respectively, by
\begin{eqnarray}
\label{eq:3.4.12}
f_{-t} & = & \cos t \ f - \sin t \ \xi = \cos (-t)f + \sin (-t) \xi,  \\
\xi_{-t} & = &  \sin t \ f + \cos t \ \xi = - \sin (-t) f + \cos (-t) \xi.\nonumber
\end{eqnarray}
The minus sign occurs because $P_t$ takes a sphere with center $f_{-t}(x)$ and radius $-t$ to the point sphere
$f_{-t}(x)$.  We call $P_t \lambda$ a {\em parallel submanifold}
of $\lambda$.  Formula (\ref{eq:3.4.12}) shows
the close correspondence between these parallel submanifolds and the 
parallel hypersurfaces $f_t$ to $f$,
in the case where $f$ is an immersed hypersurface.  

In the case where the spherical projection $f$ is an immersion at a point $x \in M^{n-1}$, we know that the number
of values of $t$ in the interval $[0,\pi)$ for which $f_t$ is not an immersion is at most $n-1$, the
maximum number of distinct principal curvatures of $f$ at $x$.  Pinkall \cite[p. 428]{P4} proved that this statement is also true for an arbitrary Legendre submanifold, even if the spherical projection $f$ is not an immersion
at $x$ by proving the following theorem (see also \cite[pp. 68--72]{Cec1} for a proof).

\begin{theorem}
\label{thm:3.4.4} 
Let $\lambda: M^{n-1} \rightarrow \Lambda^{2n-1}$ be a Legendre submanifold with spherical projection $f$ and spherical
unit normal field $\xi$.  Then for each $x \in M^{n-1}$, the parallel map, 
\begin{displaymath}
f_t  =  \cos t \ f + \sin t \ \xi,
\end{displaymath}
fails to be an immersion at $x$ for at most $n-1$ values of $t \in [0,\pi)$.
\end{theorem}

As a consequence of Pinkall's theorem, one can pass to a parallel submanifold, if necessary, to obtain the following important corollary by using well known results concerning immersed hypersurfaces in $S^n$.
Note that parts (a)--(c)  of the corollary are pointwise statements, while (d)--(e) hold on an open set $U$ 
if they can be shown to hold in a neighborhood of each point of $U$.  

\begin{corollary}
\label{cor:3.4.5} 
Let $\lambda: M^{n-1} \rightarrow \Lambda^{2n-1}$ be a Legendre submanifold. Then:
\begin{enumerate}
\item[${\rm(a)}$] at each point $x \in M^{n-1}$, there are at most $n-1$ distinct curvature spheres $K_1, \ldots, K_g$,
\item[${\rm(b)}$] the principal vectors corresponding to a curvature sphere $K_i$ form a subspace $T_i$ of the tangent space $T_xM^{n-1}$,
\item[${\rm(c)}$] the tangent space $T_xM^{n-1} = T_1 \oplus \cdots \oplus T_g$,
\item[${\rm(d)}$] if the dimension of a given $T_i$ is constant on an open subset $U$ of $M^{n-1}$, then the principal distribution $T_i$ is integrable on $U$,
\item[${\rm(e)}$] if $\dim T_i = m > 1$ on an open subset $U$ of $M^{n-1}$, then the curvature sphere map $K_i$ is constant along the leaves of the principal foliation $T_i$.
\end{enumerate}
\end{corollary}

\section{Dupin Submanifolds}
\label{sec:7a}

We now recall some basic concepts from the theory of Dupin hypersurfaces in $S^n$
(see, for example, \cite[pp. 9--35]{CR8} for more detail), and then generalize the notion of Dupin to 
Legendre submanifolds in Lie sphere geometry.

Let $f:M \rightarrow S^n$
be an immersed hypersurface, and let $\xi$ be a locally defined
field of unit normals to $f(M)$.
A {\em curvature surface} of $M$ is a smooth submanifold $S\subset M$
such that for each point $x \in S$, the tangent space
$T_xS$ is equal to a principal space (i.e., an eigenspace) of the shape operator
$A$ of $M$ at $x$. This generalizes the classical notion of a line of curvature 
for a principal curvature of multiplicity one.
The hypersurface $M$ is said to be {\em Dupin} if:

\begin{enumerate}
\item[(a)] along each curvature surface, the corresponding principal curvature is constant.
\end{enumerate}
Furthermore, a Dupin hypersurface $M$ is called {\em proper Dupin} if, in addition to Condition (a),
the following condition is satisfied:
\begin{enumerate}
\item[(b)] the number $g$ of distinct principal curvatures is constant on
$M$.
\end{enumerate}

Clearly isoparametric hypersurfaces in $S^n$ are proper Dupin, and so are
those hypersurfaces in ${\bf R}^n$ obtained from isoparametric hypersurfaces in $S^n$ via
stereographic projection (see, for example, \cite[pp. 28--30]{CR8}).  
In particular, the well-known ring cyclides of Dupin in ${\bf R}^3$ are obtained
in this way from a standard product torus $S^1(r) \times S^1(s)$ in $S^3$, where $r^2+s^2=1$.

Using the Codazzi equation, one can show that if a curvature surface $S$ has dimension greater than one,
then the corresponding principal curvature is constant on $S$.  This is
not necessarily true on a curvature surface of dimension equal to one
(i.e., a line of curvature).

Second, Condition (b) is equivalent to requiring that each continuous
principal curvature function 
has constant multiplicity on $M$. Further, for any hypersurface $M$ in $S^n$, there exists a
dense open subset of $M$ on which 
the number of distinct principal
curvatures is locally constant (see, for example, Singley \cite{Sin}).

It also follows from the Codazzi equation that if a continuous principal curvature function $\mu$ has constant
multiplicity $m$ on a connected open subset $U \subset M$, then $\mu$ is a smooth function on $U$, and
the distribution $T_{\mu}$ of principal
spaces corresponding to $\mu$ is a smooth foliation whose leaves are the curvature surfaces corresponding to 
$\mu$ on $U$. This principal curvature function
$\mu$ is constant along each of its curvature surfaces in $U$
if and only if these curvature surfaces are open subsets
of $m$-dimensional great or small spheres in $S^n$ (see \cite[pp. 24--32]{CR8}).  

We can generalize the notion of a curvature surface for hypersurfaces in
real space forms to Legendre submanifolds. Specifically, let
$\lambda: M^{n-1} \rightarrow \Lambda^{2n-1}$ be a Legendre submanifold. A connected submanifold
$S$ of $M^{n-1}$ is called a 
{\em curvature surface} if at each $x \in S$, the tangent space $T_xS$ is
equal to some principal space $T_i$, as in Corollary  \ref{cor:3.4.5}.  For example, if $\dim T_i$ 
is constant on an open subset $U$ of $M^{n-1}$,
then each leaf of the principal foliation $T_i$ is a curvature surface on $U$.

There exist many examples of Dupin hypersurfaces in $S^n$ or ${\bf R}^n$
that are not proper Dupin, because the number of 
distinct principal curvatures is not constant on the hypersurface.
This also results in curvature surfaces that are not leaves of a 
principal foliation.  An example due to Pinkall \cite{P4} is a
tube $M^3$ in ${\bf R}^4$ of constant radius over a torus of revolution 
$T^2 \subset {\bf R}^3 \subset {\bf R}^4$ (see also \cite[p. 69]{Cec1} for a description of Pinkall's example).

 One consequence of the results mentioned above
is that proper Dupin hypersurfaces in $S^n$ or ${\bf R}^n$
are algebraic, as is the case with isoparametric hypersurfaces,
as shown by M\"{u}nzner \cite{Mu}--\cite{Mu2}.
This result is most easily formulated for hypersurfaces in ${\bf R}^n$.  It states
that a connected proper Dupin hypersurface 
$f:M \rightarrow {\bf R}^n$ must be contained in a connected component of an 
irreducible algebraic subset of ${\bf R}^n$ of dimension $n-1$.  Pinkall \cite{P6} sent the author a letter in 1984 that contained a sketch of a proof of this result, but he did not publish a proof.  In 2008,
Cecil, Chi and Jensen \cite{CCJ3} used methods of real
algebraic geometry to give a proof of this result
based on Pinkall's sketch. The proof makes use of the various principal foliations whose leaves are open subsets of spheres to construct an analytic algebraic parametrization of a 
neighborhood of $f(x)$ for each point $x \in M$. 
In contrast to the situation for isoparametric hypersurfaces, however, a connected proper Dupin hypersurface 
does not necessarily lie in a compact connected proper Dupin hypersurface,
as Pinkall's example mentioned above of a tube $M^3$ in ${\bf R}^4$ of constant radius over a torus of revolution 
$T^2 \subset {\bf R}^3 \subset {\bf R}^4$ shows.

Next we generalize the definition of a Dupin hypersurface in a real space form to the setting
of Legendre submanifolds in Lie sphere geometry.
We say that a Legendre submanifold $\lambda: M^{n-1} \rightarrow \Lambda^{2n-1}$
is a {\em Dupin submanifold} if:

\begin{enumerate}
\item[(a)] along each curvature surface, the corresponding
curvature sphere map is constant.
\end{enumerate}
The Dupin submanifold $\lambda$ is called {\em proper Dupin} if, in addition 
to Condition (a), the following condition is satisfied:

\begin{enumerate}
\item[(b)] the number $g$ of distinct curvature spheres is constant on
$M$.
\end{enumerate}
In the case of the
Legendre lift $\lambda: M^{n-1} \rightarrow \Lambda^{2n-1}$ of an immersed
Dupin hypersurface $f:M^{n-1} \rightarrow S^n$, the submanifold $\lambda$
is a Dupin submanifold, since a curvature sphere map of $\lambda$
is constant along a curvature surface if and only if the corresponding principal curvature map of $f$
is constant along that curvature surface. Similarly, $\lambda$ is proper Dupin if and only if $f$ is proper Dupin,
since the number of distinct curvatures spheres of $\lambda$ at a point $x \in M^{n-1}$ equals the number
of distinct principal curvatures of $f$ at $x$. Particularly important examples of proper Dupin submanifolds
are the Legendre lifts of isoparametric hypersurfaces in $S^n$.

We now show that Theorem~\ref{thm:3.4.3} implies that both the Dupin and proper Dupin conditions are invariant under Lie sphere transformations.
Many important classification results for Dupin submanifolds have been obtained in the setting of
Lie sphere geometry (see Chapter 5 of  \cite{Cec1}).

\begin{theorem}
\label{thm:dupin-lie-invariant} 
Let $\lambda: M^{n-1} \rightarrow \Lambda^{2n-1}$ be a Legendre submanifold and $\beta$ a Lie sphere transformation.
\begin{enumerate}
\item[${\rm(a)}$] If $\lambda$ is Dupin, then $\beta \lambda$ is Dupin.
\item[${\rm(b)}$] If $\lambda$ is proper Dupin, then $\beta \lambda$ is proper Dupin.
\end{enumerate}
\end{theorem}

\begin{proof}
By Theorem~\ref{thm:3.4.3}, a point $[K]$ on the line $\lambda (x)$ is a curvature sphere of $\lambda$ at $x \in M$
if and only if the point $\beta [K]$ is a curvature sphere of $\beta \lambda$ at $x$, and the principal
spaces corresponding to $[K]$ and  $\beta [K]$ are identical.  Since these principal spaces are the same, if $S$ is
a curvature surface of $\lambda$ corresponding to a curvature sphere map $[K]$, then $S$ is also
a curvature surface of $\beta \lambda$ corresponding to a curvature sphere map $\beta [K]$, and clearly
$[K]$ is constant along $S$ if and only if $\beta [K]$ is constant along $S$.  This proves part (a) of the theorem.
Part (b) also follows immediately from Theorem~\ref{thm:3.4.3}, since for each $x \in M$, the number $g$ of
distinct curvature spheres of $\lambda$ at $x$ equals the number of distinct curvatures spheres of $\beta \lambda$
at $x$.  So if this number $g$ is constant on $M$ for $\lambda$, then it is constant on $M$ for $\beta \lambda$.
\end{proof}

\section{Lifts of Isoparametric Hypersurfaces}
\label{sec:9}
In this section, we give a Lie sphere geometric characterization of the Legendre lifts of isoparametric
hypersurfaces in the sphere $S^n$ (Theorem \ref{thm:3.5.6}).  This result has been used in several papers to
prove that under certain conditions a proper Dupin submanifold is Lie equivalent to the Legendre
lift of an isoparametric hypersurface.

Let $\lambda: M^{n-1} \rightarrow \Lambda^{2n-1}$ be an arbitrary Legendre submanifold.  As before, we can write
$\lambda = [Y_1, Y_{n+3}]$, where
\begin{equation}
\label{eq:3.5.1}
Y_1 = (1, f, 0), \quad Y_{n+3} = (0, \xi, 1),
\end{equation} 
where $f$ and $\xi$ are the spherical projection and spherical field of unit normals, 
respectively.

For $x \in M^{n-1}$, the points on the line $\lambda(x)$ can be written in the form,
\begin{equation}
\label{eq:3.5.2}
\mu Y_1 (x) +  Y_{n+3} (x),
\end{equation} 
that is, we take $\mu$ as an inhomogeneous coordinate along the projective line $\lambda(x)$. Then 
the point sphere $[Y_1]$
corresponds to $\mu = \infty$.  The next two theorems give the relationship between the coordinates of the 
curvature spheres of $\lambda$ and the 
principal curvatures of $f$, in the case where 
$f$ has constant rank. In the first
theorem, we assume that the spherical projection $f$ is an immersion on $M^{n-1}$.  By Theorem~\ref{thm:3.4.4},
we know that this can always be achieved locally by passing to a parallel submanifold.

\begin{theorem}
\label{thm:3.5.1} 
Let $\lambda: M^{n-1} \rightarrow \Lambda^{2n-1}$ be a Legendre submanifold whose spherical projection 
$f:M^{n-1} \rightarrow S^n$ is an immersion.  Let $Y_1$ and $Y_{n+3}$ be the point sphere and great sphere maps
of $\lambda$ as in equation (\ref{eq:3.5.1}).  Then the curvature spheres of $\lambda$ at a point $x \in M^{n-1}$ are
\begin{displaymath}
[K_i] = [\kappa_i Y_1 + Y_{n+3}], \quad 1 \leq i \leq g,
\end{displaymath}
where $\kappa_1,\ldots,\kappa_g$ are the distinct principal curvatures at $x$ of the oriented hypersurface $f$
with field of unit normals $\xi$.  The multiplicity of the curvature sphere $[K_i]$ equals the multiplicity
of the principal curvature $\kappa_i$.
\end{theorem}
\begin{proof}
Let $X$ be a nonzero vector in $T_xM^{n-1}$.  Then for any real number $\mu$,
\begin{displaymath}
d(\mu Y_1  +  Y_{n+3})(X) = (0, \mu \ df(X) + d\xi(X), 0). 
\end{displaymath}
This vector is in Span $\{Y_1(x), Y_{n+3}(x)\}$ if and only if
\begin{displaymath}
\mu \ df(X) + d\xi(X) = 0, 
\end{displaymath}
i.e., $\mu$ is a principal curvature of $f$ with corresponding principal vector $X$.
\end{proof}

We next consider the case where the point sphere map $Y_1$ is a curvature sphere of constant multiplicity $m$ on
$M^{n-1}$.  By Corollary \ref{cor:3.4.5}, the corresponding principal distribution is a foliation, and the 
curvature sphere map $[Y_1]$ is constant along the leaves of this foliation.  Thus the map $[Y_1]$ factors through
an immersion $[W_1]$ from the space of leaves
$V$ of this foliation into $Q^{n+1}$.  We can write $[W_1] = [(1, \phi, 0)]$,
where $\phi:V \rightarrow S^n$ is an immersed submanifold of codimension $m+1$. The manifold $M^{n-1}$ is locally diffeomorphic to an open subset of the unit normal 
bundle $B^{n-1}$ of the submanifold $\phi$, and $\lambda$ is essentially
the Legendre lift of $\phi (V)$, as defined in Section \ref{sec:6}.  
The following theorem relates the curvature spheres
of $\lambda$ to the principal curvatures of $\phi$.  Recall that the point sphere and great sphere maps for $\lambda$
are given as in equation (\ref{eq:3.3.3}) by
\begin{equation}
\label{eq:3.5.3}
Y_1(x,\xi) = (1, \phi(x), 0), \quad Y_{n+3}(x,\xi) = (0, \xi, 1).
\end{equation}
 
\begin{theorem}
\label{thm:3.5.2} 
Let $\lambda: B^{n-1} \rightarrow \Lambda^{2n-1}$ be the Legendre lift
of an immersed submanifold $\phi(V)$ in $S^n$
of codimension $m+1$.  Let $Y_1$ and $Y_{n+3}$ be the point sphere and great sphere maps
of $\lambda$ as in equation (\ref{eq:3.5.3}).  Then the curvature spheres of $\lambda$ at a point $(x, \xi) \in B^{n-1}$ are
\begin{displaymath}
[K_i] = [\kappa_i Y_1 + Y_{n+3}], \quad 1 \leq i \leq g,
\end{displaymath}
where $\kappa_1,\ldots,\kappa_{g-1}$ are the distinct principal curvatures of the shape operator $A_\xi$,
and $\kappa_g = \infty$.  For $1 \leq i \leq g-1$, the multiplicity of the curvature sphere $[K_i]$ equals the multiplicity
of the principal curvature $\kappa_i$, while the multiplicity of $[K_g]$ is $m$.
\end{theorem}

The proof of this theorem is similar to that of Theorem \ref{thm:3.5.1}, but one must introduce local coordinates on the unit normal bundle to get a complete proof (see \cite[p. 74]{Cec1}).

We close this section with a local Lie geometric characterization of 
Legendre submanifolds that are Lie equivalent
to the Legendre lift of
an isoparametric hypersurface in $S^n$ (see
 \cite{Cec4} or \cite[p. 77]{Cec1}).  Here a line in ${\bf P}^{n+2}$ is called 
{\em timelike} if it contains only timelike points.
This means that an orthonormal basis for the 2-plane in ${\bf R}^{n+3}_2$ determined by the timelike line consists
of two timelike vectors.  An example is the line $[e_1, e_{n+3}]$.  This theorem has been useful in obtaining several classification results for proper Dupin hypersurfaces.

\begin{theorem}
\label{thm:3.5.6} 
Let $\lambda: M^{n-1} \rightarrow \Lambda^{2n-1}$ be a Legendre submanifold with $g$ distinct curvature spheres
$[K_1],\ldots,[K_g]$ at each point.  Then $\lambda$ is Lie equivalent 
to the Legendre lift of an isoparametric
hypersurface in $S^n$ if and only if there exist $g$ points $[P_1],\ldots,[P_g]$ on a timelike line in ${\bf P}^{n+2}$
such that
\begin{displaymath}
\langle K_i,P_i \rangle = 0, \quad 1 \leq i \leq g.
\end{displaymath}
\end{theorem}
\begin{proof}
If $\lambda$ is the Legendre lift of an 
isoparametric hypersurface in $S^n$, then all the spheres in a family
$[K_i]$ have the same radius 
$\rho_i$, where $0 < \rho_i < \pi$.  By formula (\ref{eq:1.4.4}),
this is equivalent to the condition $\langle K_i,P_i \rangle = 0$, where
\begin{equation}
\label{eq:3.5.7}
P_i = \sin \rho_i \ e_1 - \cos \rho_i \ e_{n+3}, \quad 1 \leq i \leq g,
\end{equation}
are $g$ points on the timelike line $[e_1, e_{n+3}]$ (see \cite[pp. 17--18]{Cec1}).  
Since a Lie sphere transformation preserves curvature spheres,
timelike lines and the polarity relationship, the same is true for any image of $\lambda$
under a Lie sphere transformation.

Conversely, suppose that there exist $g$ points $[P_1],\ldots,[P_g]$ on a timelike line $\ell$ such that
$\langle K_i,P_i \rangle = 0$, for $1 \leq i \leq g$.
Let $\beta$ be a Lie sphere transformation that maps $\ell$
to the line $[e_1, e_{n+3}]$.  Then the curvature spheres $\beta [K_i]$ of $\beta \lambda$ are 
orthogonal to the points $[Q_i] = \beta [P_i]$ on the line $[e_1, e_{n+3}]$.  
By (\ref{eq:1.4.4}), this means that the spheres corresponding
to $\beta [K_i]$ have constant radius on $M^{n-1}$.  By applying a 
parallel transformation $P_t$, if necessary, we can arrange 
that none of these curvature spheres has radius zero. 
Then $P_t \beta \lambda$ is the Legendre lift of an isoparametric hypersurface in $S^n$.
\end{proof}

\section{Cyclides of Dupin}
\label{sec:10}
The classical cyclides of Dupin in ${\bf R}^3$ were studied intensively by many leading
mathematicians in the nineteenth century, including Liouville \cite{Lio}, 
Cayley \cite{Cay}, and Maxwell \cite{Max}, 
whose paper contains stereoscopic
figures of the various types of cyclides. 
A good account of the history of the cyclides in the nineteenth century is given by
Lilienthal \cite{Lil} (see also
Klein \cite[pp. 56--58]{Klein}, Darboux \cite[vol. 2, pp. 267--269]{Darboux},
Blaschke \cite[p. 238]{Bl}, Eisenhart \cite[pp. 312--314]{Eisenhart}, 
Hilbert and Cohn-Vossen \cite[pp. 217--219]{HC-V}, Fladt and Baur \cite[pp. 354--379]{FB}, Banchoff \cite{Ban1},
and Cecil and Ryan \cite[pp. 151--166]{CR7}).

We now turn our attention to Pinkall's classification of the cyclides of Dupin of arbitrary dimension, 
which is obtained by using the methods of Lie sphere geometry.
Our presentation here is based on the accounts of this subject given in
 \cite[pp. 148--159]{Cec1} and \cite[pp. 263--283]{CR8}.
A proper Dupin submanifold $\lambda: M^{n-1} \rightarrow \Lambda^{2n-1}$ with two distinct curvature spheres
of respective multiplicities $p$ and $q$ at each point is called a
{\em cyclide of Dupin of characteristic} $(p,q)$.  

We will prove that any connected cyclide of Dupin of
characteristic $(p,q)$ is contained in a unique compact, connected cyclide of Dupin of
characteristic $(p,q)$.  Furthermore, every compact, connected cyclide of Dupin of
characteristic $(p,q)$ is Lie equivalent to the Legendre lift of a standard 
product of two spheres,
\begin{equation}
\label{eq:std-prod-p-q}
S^q (1/\sqrt{2}) \times S^p (1/\sqrt{2}) \subset S^n \subset {\bf R}^{q+1} \times {\bf R}^{p+1} = {\bf R}^{n+1}, 
\end{equation} 
where $p$ and $q$ are positive integers such that $p+q = n-1$. Thus any two compact,
connected cyclides of Dupin of the same characteristic are Lie equivalent.

It is well known that the product $S^q (1/\sqrt{2}) \times S^p (1/\sqrt{2})$ is an isoparametric hypersurface
in $S^n$ with two distinct principal curvatures having multiplicities $m_1 = p$ and $m_2 = q$ (see, for
example, \cite[pp. 110--111]{CR8}).  
Furthermore, every compact
isoparametric hypersurface in $S^n$ with two principal curvatures of multiplicities $p$ and $q$ is Lie equivalent
to $S^q (1/\sqrt{2}) \times S^p (1/\sqrt{2})$, since it is congruent to a parallel hypersurface of 
$S^q (1/\sqrt{2}) \times S^p (1/\sqrt{2})$.

Although $S^q (1/\sqrt{2}) \times S^p (1/\sqrt{2})$ is a good model for the cyclides, it is often easier
to work with the two focal submanifolds $S^q (1)\times \{0\}$ and $\{0\} \times S^p (1)$ in proving classification
results.  The Legendre lifts of these two focal submanifolds are Lie equivalent to the Legendre lift of
$S^q (1/\sqrt{2}) \times S^p (1/\sqrt{2})$, since they are parallel submanifolds of the Legendre lift of
$S^q (1/\sqrt{2}) \times S^p (1/\sqrt{2})$.  In fact, the hypersurface $S^q (1/\sqrt{2}) \times S^p (1/\sqrt{2})$
is a tube of spherical radius $\pi/4$ in $S^n$ over either of its two focal submanifolds.

We now describe our standard model of a cyclide of characteristic $(p,q)$ in the context
of Lie sphere geometry, as in Pinkall's paper \cite{P4} (see also \cite[p. 149]{Cec1}).
Let $\{e_1,\ldots,e_{n+3} \}$ be the standard orthonormal basis for ${\bf R}^{n+3}_2$,
with $e_1$ and $e_{n+3}$ unit timelike vectors, and $\{e_2,\ldots,e_{n+2} \}$ unit spacelike vectors.
Then $S^n$ is the unit sphere in the Euclidean space ${\bf R}^{n+1}$ spanned by $\{e_2,\ldots,e_{n+2} \}$.
Let
\begin{equation}
\label{eq:4.3.1}
\Omega = {\mbox {\rm Span }}\{e_1,\ldots,e_{q+2} \}, \quad \Omega^{\perp} = {\mbox {\rm Span }}\{e_{q+3},\ldots,e_{n+3} \}.
\end{equation}
These spaces have signatures
$(q+1,1)$ and $(p+1,1)$, respectively.  The intersection $\Omega \cap Q^{n+1}$
is the quadric given in homogeneous coordinates by 
\begin{equation}
\label{eq:4.3.1a}
x_1^2 = x_2^2 + \cdots + x_{q+2}^2, \quad x_{q+3} = \cdots = x_{n+3} = 0.
\end{equation}
This set is diffeomorphic to the unit sphere $S^q$ in
\begin{displaymath}
{\bf R}^{q+1} = {\mbox {\rm Span }}\{e_2,\ldots,e_{q+2} \},
\end{displaymath}
by the diffeomorphism $\phi : S^q \rightarrow \Omega \cap Q^{n+1}$, defined by $\phi (v) = [e_1 + v]$.  

Similarly, the quadric $\Omega^{\perp} \cap Q^{n+1}$ is diffeomorphic to the unit sphere $S^p$ in
\begin{displaymath}
{\bf R}^{p+1} = {\mbox {\rm Span }}\{e_{q+3},\ldots,e_{n+2} \}
\end{displaymath}
by the diffeomorphism $\psi : S^p \rightarrow \Omega^{\perp} \cap Q^{n+1}$ defined by $\psi (u) = [u+e_{n+3}]$.  

The model that we will use for the cyclides in Lie sphere geometry is the Legendre
submanifold $\lambda : S^p \times S^q \rightarrow \Lambda^{2n-1}$ defined by 
\begin{equation}
\label{eq:4.3.2}
\lambda (u,v) = [k_1,k_2], {\mbox {\rm with }} [k_1 (u,v)] = [\phi(v)], \quad [k_2 (u,v)] = [\psi(u)].
\end{equation}
It is easy to check that the Legendre Conditions $(1)$--$(3)$ of Theorem \ref{thm:3.2.3} 
are satisfied by the pair $\{ k_1,k_2 \}$.  To find
the curvature spheres of $\lambda$, we decompose the tangent space to $S^p \times S^q$ at a point $(u,v)$ as
\begin{displaymath}
T_{(u,v)} S^p \times S^q = T_u S^p \times T_v S^q.
\end{displaymath}
Then $dk_1 (X,0) = 0$ for all $X \in T_u S^p$, and $dk_2 (Y) = 0$ for all $Y$ in $T_v S^q$. 
Thus, $[k_1]$ and
$[k_2]$ are curvature spheres of $\lambda$ with respective multiplicities $p$ and $q$.  Furthermore, the image
of $[k_1]$ lies in the quadric $\Omega \cap Q^{n+1}$, and the image of $[k_2]$ is contained in 
the quadric $\Omega^{\perp} \cap Q^{n+1}$.
The point sphere map of $\lambda$ is $[k_1]$, and thus $\lambda$ is the Legendre lift of the focal submanifold
$S^q \times \{0\} \subset S^n$, considered as a submanifold of codimension $p+1$ in $S^n$.  As noted above, this Legendre lift $\lambda$ of the focal submanifold is Lie equivalent to the
Legendre lift of the standard product of spheres by means of a parallel transformation.

We now prove Pinkall's \cite{P4} 
classification of proper Dupin submanifolds with two distinct curvature spheres 
at each point.
Pinkall's proof depends on establishing the existence of a local principal coordinate system.
This can always be done in the case of $g=2$ curvature spheres, because the 
the sum of the dimensions of the two principal spaces is $n-1$, the dimension of $M^{n-1}$
(see, for example,  \cite[p. 249]{CR8}).   Such a coordinate system might not exist in the case $g>2$.
In fact, if $M$ is an isoparametric hypersurface in $S^n$
with more than two distinct principal curvatures, then there cannot exist a
local principal coordinate system on $M$ (see, for example, \cite[pp. 180--184]{CR7} or
\cite[pp. 248--252]{CR8}).

For a different proof of Pinkall's theorem (Theorem 10.1 below) using the method of moving frames, see the paper of Cecil-Chern \cite{CC2} or \cite[pp. 266--273]{CR8}.
That approach generalizes to the study of proper Dupin hypersurfaces with $g>2$ curvature spheres (see, for example, Cecil and Jensen \cite{CJ2}, \cite{CJ3}).

Note that before Pinkall's paper, Cecil and Ryan \cite{CR2}--\cite{CR5} (see also \cite[pp. 166--179]{CR7})
proved a classification of complete cyclides in $ {\bf R}^n$ using techniques of Euclidean submanifold theory.  However, the proof used the assumption of completeness in an essential way, and that theorem did not contain part (a) of Pinkall's Theorem 10.1 below.

\begin{theorem}
\label{thm:4.3.1}
\begin{enumerate}
\item[${\rm(a)}$] Every connected cyclide of Dupin of characteristic $(p,q)$ 
is contained in a unique compact, connected cyclide of Dupin characteristic $(p,q)$.
\item[${\rm(b)}$] Every compact, connected cyclide of Dupin of characteristic $(p,q)$ is Lie equivalent to the Legendre lift of a standard product of two spheres
\begin{equation}
\label{eq:std-prod-p-q-2}
S^q (1/\sqrt{2}) \times S^p (1/\sqrt{2}) \subset S^n \subset {\bf R}^{q+1} \times {\bf R}^{p+1} = {\bf R}^{n+1}, 
\end{equation} 
where $p+q = n-1$.  Thus, any two compact, connected cyclides of Dupin of characteristic $(p,q)$ are
Lie equivalent.
\end{enumerate}
\end{theorem}
\begin{proof}
Suppose that $\lambda: M^{n-1} \rightarrow \Lambda^{2n-1}$ is a connected
cyclide of Dupin of characteristic $(p,q)$ with $p+q = n-1$.  We may take $\lambda = [k_1, k_2]$, where $[k_1]$
and $[k_2]$ are the curvature spheres with respective multiplicities $p$ and $q$.  Each curvature sphere map
factors through an immersion of the space of leaves of its principal foliation $T_i$ for $i = 1,2$.  Since the sum of
the dimensions of  $T_1$ and $T_2$ equals the dimension of $M^{n-1}$, locally we
can take a principal coordinate system $(u,v)$ (see, for example,  \cite[p. 249]{CR8}) defined on an open set
\begin{displaymath}
W = U \times V \subset {\bf R}^p \times {\bf R}^q,
\end{displaymath} 
such that\\

\noindent
(i) $[k_1]$ depends only on $v$, and $[k_2]$ depends only on $u$, for all $(u,v) \in W$.\\

\noindent
(ii) $[k_1(W)]$ and $[k_2(W)]$ are submanifolds of $Q^{n+1}$ of dimensions $q$ and $p$, respectively.\\

\noindent
Now let $(u,v)$ and $(\bar{u},\bar{v})$ be any two points in $W$.  From (i), we have
the following key equation,
\begin{equation}
\label{eq:4.3.3}
\langle k_1(u,v), k_2 (\bar{u},\bar{v}) \rangle  = \langle k_1(v), k_2 (\bar{u}) \rangle =
\langle k_1(\bar{u},v), k_2 (\bar{u},v) \rangle = 0,
\end{equation}
since $[k_1]$ and $[k_2]$ are orthogonal at every point $x \in M^{n-1}$, in particular $x = (\bar{u},v)$.

Let $E$ be the smallest linear subspace of ${\bf P}^{n+2}$ containing the $q$-dimensional submanifold 
$[k_1(W)]$.
By equation (\ref{eq:4.3.3}), we have
\begin{equation}
\label{eq:4.3.4}
[k_1(W)] \subset E \cap Q^{n+1}, \quad [k_2(W)] \subset E^{\perp} \cap Q^{n+1}.
\end{equation}
The dimensions of $E$ and $E^{\perp}$ as subspaces of ${\bf P}^{n+2}$ satisfy
\begin{equation}
\label{eq:4.3.5}
\dim E + \dim E^{\perp} = n+1 = p+q+2.
\end{equation}
We claim that $\dim E = q+1$ and $\dim E^{\perp} = p+1$.  

To see this, suppose first that $\dim E > q+1$.
Then $\dim E^{\perp} \leq p$, and $E^{\perp} \cap Q^{n+1}$ cannot contain the $p$-dimensional 
submanifold $k_2 (W)$, contradicting equation (\ref{eq:4.3.4}).
Similarly, assuming that $\dim E^{\perp} > p+1$ leads to a contradiction, since  
then $\dim E \leq q$, and $E \cap Q^{n+1}$ cannot contain the $q$-dimensional 
submanifold $k_1 (W)$.

Thus we have
\begin{displaymath}
\dim E \leq q+1, \quad \dim E^{\perp} \leq p+1.
\end{displaymath}
This and equation (\ref{eq:4.3.5}) imply that $\dim E = q+1$ and $\dim E^{\perp} = p+1$.
Furthermore, from the fact that $E \cap Q^{n+1}$ and $E^{\perp} \cap Q^{n+1}$ contain submanifolds of dimensions
$q$ and $p$, respectively, it is easy to deduce that the Lie inner product $\langle \  ,\  \rangle$ has signature
$(q+1,1)$ on $E$ and $(p+1,1)$ on $E^{\perp}$.  

Take an orthornormal basis $\{w_1,\ldots,w_{n+3}\}$ of ${\bf R}^{n+3}_2$ with $w_1$ and $w_{n+3}$
timelike such that 
\begin{equation}
\label{eq:CC-4.31d}
E = {\mbox {\rm Span }}\{w_1,\ldots,w_{q+2}\}, \quad E^{\perp} = 
{\mbox {\rm Span }}\{w_{q+3},\ldots,w_{n+3} \}.
\end{equation}
Then $E \cap Q^{n+1}$ is given in homogeneous coordinates $( x_1,\ldots,x_{n+3})$ with respect to this basis by
\begin{equation}
\label{eq:CC-4.32}
x_1^2 = x_2^2 + \cdots + x_{q+2}^2, \quad  x_{q+3} = \cdots = x_{n+3} = 0.
\end{equation}
This quadric is diffeomorphic to the unit sphere $S^q$ in the span ${\bf R}^{q+1}$ of the spacelike vectors 
$w_2,\ldots,w_{q+2}$ with the diffeomorphism $\gamma: S^q \rightarrow E \cap Q^{n+1}$ given by
\begin{equation}
\label{eq:CC-4.33}
\gamma (v) = [w_1 + v], \quad v \in S^q.
\end{equation}
Similarly $E^{\perp} \cap Q^{n+1}$ is the quadric given in homogeneous coordinates by
\begin{equation}
\label{eq:CC-4.34}
x_{n+3}^2 = x_{q+3}^2 + \cdots + x_{n+2}^2, \quad  x_1 = \cdots = x_{q+2} = 0.
\end{equation}
This space $E^{\perp} \cap Q^{n+1}$ is diffeomorphic to the unit sphere $S^p$ in the span
${\bf R}^{p+1}$ of the spacelike vectors $w_{q+3},\ldots,w_{n+2}$ with the 
diffeomorphism 
\begin{displaymath}
\delta: S^p \rightarrow E^{\perp} \cap Q^{n+1}
\end{displaymath}
given by
\begin{equation}
\label{eq:CC-4.35}
\delta (u) = [u + w_{n+3}], \quad u \in S^p.
\end{equation}

The image of the curvature sphere map $k_1$ of multiplicity $p$
is contained in the $q$-dimensional quadric $E \cap Q^{n+1}$ given by
equation (\ref{eq:CC-4.32}), which is diffeomorphic to $S^q$.  
The map $k_1$ is constant on each leaf of its principal foliation $T_1$,
and so $k_1$ factors through an immersion of the $q$-dimensional space of leaves $W/T_1$ into the 
$q$-dimensional quadric $E \cap Q^{n+1}$.
Hence, the image of $k_1$ is an open subset of this quadric, and each leaf of $T_1$ corresponds to a point
$\gamma (v)$ of the quadric.
 
Similarly, the curvature sphere map $k_2$ of multiplicity $q$ factors through an immersion
of its $p$-dimensional space of leaves $W/T_2$ onto an open subset of the $p$-dimensional
quadric $E^{\perp} \cap Q^{n+1}$ given by 
equation (\ref{eq:CC-4.34}), and each leaf of $T_2$ corresponds to a point
of $\delta(u)$ of that quadric.  

From this it is clear that the restriction of 
the Legendre map $\lambda$ to the neighborhood $W \subset M$ is contained
in the compact, connected cyclide 
\begin{displaymath}
\nu : S^p \times S^q \rightarrow \Lambda^{2n-1}
\end{displaymath}
defined by
\begin{equation}
\label{eq:CC-4.35a}
\nu (u,v) = [k_1 (u,v), k_2 (u,v)], \quad (u,v) \in S^p \times S^q,
\end{equation}
where
\begin{equation}
\label{eq:CC-4.35b}
k_1 (u,v) = \gamma (v), \quad k_2 (u,v) = \delta (u),
\end{equation}
for the maps $\gamma$ and $\delta$ defined above. By a standard connectedness argument,
the Legendre map $\lambda: M \rightarrow \Lambda^{2n-1}$ is 
also the restriction of the compact, connected cyclide $\nu$ to an open subset of $S^p \times S^q$.
This proves part (a) of the theorem.

In projective space ${\bf P}^{n+2}$, 
the image of $\nu$ consists of all lines joining a point on the quadric $E \cap Q^{n+1}$ in equation (\ref{eq:CC-4.32}) to a point on the quadric $E^{\perp} \cap Q^{n+1}$ in equation (\ref{eq:CC-4.34}). 
Thus any choice of a $(q+1)$-plane $E$ in 
${\bf P}^{n+2}$ with signature $(q+1,1)$ and corresponding orthogonal complement $E^{\perp}$
with signature $(p+1,1)$ determines a unique compact, connected cyclide of characteristic $(p,q)$ and
vice-versa. 

The Lie equivalence of any two compact, connected cyclides of the same characteristic stated in 
part (b) of the theorem is then clear,
since given any two choices $E$ and $F$ of $(q+1)$-planes in 
${\bf P}^{n+2}$ with signature $(q+1,1)$, there is a Lie sphere transformation that maps $E$ to $F$ and
$E^{\perp}$ to $F^{\perp}$.

In particular, if we take $F$ to be the space $\Omega$ in equation (\ref{eq:4.3.1}), then the corresponding 
cyclide is our standard model.  So our given compact, connected cyclide $\nu$ 
in equation (\ref{eq:CC-4.35a}) is Lie equivalent to our standard model.
As noted before the statement of Theorem \ref{thm:4.3.1}, our standard model is Lie equivalent to the
Legendre lift of the standard product of two spheres,
\begin{equation}
\label{eq:std-prod-p-q-2a}
S^q (1/\sqrt{2}) \times S^p (1/\sqrt{2}) \subset S^n \subset {\bf R}^{q+1} \times {\bf R}^{p+1} = {\bf R}^{n+1}, 
\end{equation} 
where $p+q = n-1$, via parallel transformation.  Thus, 
any compact, connected cyclide of Dupin of characteristic $(p,q)$ is Lie equivalent to the Legendre lift of a standard product of two spheres given in equation (\ref{eq:std-prod-p-q-2a}).
\end{proof}

\begin{remark}
\label{rem:10.1}
{\rm We can also see that the submanifold $\lambda$ in Theorem 
\ref{thm:4.3.1} is Lie equivalent to the Legendre lift of an isoparametric 
hypersurface in $S^n$ with two principal curvatures
by invoking Theorem \ref{thm:3.5.6}, because the two curvature sphere maps $[k_1]$ and $[k_2]$
are orthogonal to the two points $w_{n+3}$ and  $w_1$, respectively,
on the timelike line $[w_1, w_{n+3}]$ in ${\bf P}^{n+2}$.}      
\end{remark}

Theorem \ref{thm:4.3.1} is a classification of the cyclides of Dupin in the context of Lie sphere geometry.  It is also useful to have a M\"{o}bius geometric classification of the cyclides of Dupin
$M^{n-1} \subset {\bf R}^n$.
This is analogous to the classical characterizations of the cyclides of Dupin in ${\bf R}^3$ obtained in the nineteenth century (see, for example, \cite[pp. 151--166]{CR7}).  
K. Voss 
\cite{Voss} announced the classification in Theorem
\ref{thm:4.3.2} below for the higher-dimensional cyclides, but he did not publish a proof.  The theorem follows
quite directly from Theorem \ref{thm:4.3.1} and known results on surfaces of revolution.  

The theorem 
is phrased in terms embedded hypersurfaces in ${\bf R}^n$.  Thus we are excluding the standard model given in 
equation (\ref{eq:4.3.2}), where the spherical projection (and thus the Euclidean projection)
is not an immersion.  Of course, the spherical projections of all
parallel submanifolds of the standard model in the sphere are embedded isoparametric hypersurfaces
in the sphere $S^n$,
except for the Legendre lift of the other focal submanifold.  The proof of the following theorem
using techniques of Lie and 
M\"{o}bius geometry was given in \cite{Cec5},
and later versions of the proof
together with computer graphic illustrations of the cyclides are given in \cite[pp. 151--159]{Cec1}
and \cite[pp. 273--281]{CR8}. These proofs use same notation that we have been using in this section. 
We omit the proof here and refer the reader to these two references.

\begin{theorem}
\label{thm:4.3.2}
${\rm(a)}$ Every connected cyclide of Dupin $M^{n-1} \subset {\bf R}^n$ of characteristic $(p,q)$ is M\"{o}bius
equivalent to an open subset of a hypersurface of revolution obtained by revolving a $q$-sphere 
$S^q \subset {\bf R}^{q+1} \subset {\bf R}^n$ about an axis ${\bf R}^q \subset {\bf R}^{q+1}$ or a $p$-sphere
$S^p \subset {\bf R}^{p+1} \subset {\bf R}^n$ about an axis ${\bf R}^p \subset {\bf R}^{p+1}$.\\

\noindent
${\rm(b)}$ Two hypersurfaces obtained by revolving a $q$-sphere 
$S^q \subset {\bf R}^{q+1} \subset {\bf R}^n$ about an axis of revolution ${\bf R}^q \subset {\bf R}^{q+1}$ are M\"{o}bius equivalent if and only if they have the same value of
$\rho = |r|/a$, where $r$ is the signed radius
of the profile sphere $S^q$  and $a>0$ is the distance from the center of $S^q$ to the axis of revolution.
\end{theorem}

\begin{remark}
\label{rem:10.2}
{\rm Note that in this theorem, the profile sphere $S^q \subset {\bf R}^{q+1} \subset {\bf R}^n$
is allowed to intersect the axis of revolution  ${\bf R}^q \subset {\bf R}^{q+1}$, in which case the
hypersurface of revolution has singularities in ${\bf R}^n$.  Under M\"{o}bius transformation,
this leads to cyclides which have Euclidean singularities, such as the classical horn cyclides and
spindle cyclides (see, for example,  \cite[pp. 151--159]{Cec1}
for more detail).  In such cases,
however, the corresponding Legendre map $\lambda: S^p \times S^q \rightarrow \Lambda^{2n-1}$
is still an immersion.} 
\end{remark}

\section{Local Constructions}
\label{sec:11}
Pinkall \cite{P4} introduced four constructions 
for obtaining a Dupin hypersurface $W$ in ${\bf R}^{n+m}$
from a Dupin hypersurface $M$ in ${\bf R}^n$.  We first describe these constructions in the case $m=1$ as follows.

Begin with a Dupin hypersurface
$M^{n-1}$ in ${\bf R}^n$ and then consider ${\bf R}^n$
as the linear subspace ${\bf R}^n \times \{ 0 \}$
in ${\bf R}^{n+1}$.  The following 
constructions yield a Dupin hypersurface $W^n$ in
${\bf R}^{n+1}$.
\begin{enumerate}
\item[(1)] Let $W^n$ be the cylinder $M^{n-1} \times {\bf R}$ in
${\bf R}^{n+1}$.
\item[(2)] Let $W^n$ be the hypersurface in ${\bf R}^{n+1}$
obtained by rotating
$M^{n-1}$ around an axis (a linear subspace)
${\bf R}^{n-1} \subset {\bf R}^n$.
\item[(3)] Let $W^n$ be a tube of constant radius in ${\bf R}^{n+1}$ around $M^{n-1}$.
\item[(4)] Project $M^{n-1}$ stereographically onto a hypersurface
$V^{n-1} \subset S^n \subset {\bf R}^{n+1}$.  Let
$W^n$ be the cone over $V^{n-1}$ in ${\bf R}^{n+1}$.
\end{enumerate}

In general, these constructions introduce a new principal curvature 
of multiplicity one which is constant along its lines 
of curvature.  The other principal curvatures are determined by the 
principal curvatures of $M^{n-1}$, and the Dupin property is preserved
for these principal curvatures.  These constructions can be
generalized to produce a new principal curvature of multiplicity
$m$ by considering ${\bf R}^n$ as a subset of ${\bf R}^n \times
{\bf R}^m$ rather than ${\bf R}^n \times {\bf R}$.
(See \cite[pp. 125--148]{Cec1} for a detailed description of these constructions in full generality
in the context of Lie sphere geometry.)

Although Pinkall gave these four constructions, his Theorem 4 \cite[p. 438]{P4} 
showed that the cone construction is redundant, since it
is Lie equivalent 
to a tube. (See the proof of Theorem \ref{thm:4.2.9} and
Remark \ref{cone-construction} below.)
For this reason, we will only
study three standard constructions: tubes, cylinders and surfaces of revolution.

A Dupin submanifold obtained from a lower-dimensional Dupin submanifold via one of these standard constructions is said
to be {\em reducible}.  
More generally, a Dupin submanifold which is locally Lie equivalent to such a Dupin submanifold 
is called {\em reducible}.

Using these constructions, Pinkall was able to produce a proper Dupin hypersurface 
in Euclidean space with an
arbitrary number of distinct principal curvatures, each with any given multiplicity (see Theorem \ref{thm:4.1.1}
below).  In general, these proper Dupin hypersurfaces 
cannot be extended to compact Dupin hypersurfaces without losing the
property that the number of distinct principal curvatures is constant, as we will discuss after the proof of the theorem.  

\begin{theorem}
\label{thm:4.1.1} 
Given positive integers $m_1,\ldots,m_g$ with 
\begin{displaymath}
m_1 + \cdots + m_g = n-1,
\end{displaymath}
there exists a proper Dupin hypersurface
in ${\bf R}^n$ with $g$ distinct principal curvatures having respective multiplicities $m_1,\ldots,m_g$.
\end{theorem}
\begin{proof}
The proof is by an inductive construction, which will be clear once the first few examples are done.  To begin, note
that a usual torus of revolution 
in ${\bf R}^3$ is a proper Dupin hypersurface with two principal curvatures.  To
construct a proper Dupin hypersurface $M^3$ in ${\bf R}^4$ with three principal curvatures, each of multiplicity one,
begin with an open subset $U$ of a torus of revolution in ${\bf R}^3$ on which neither principal curvature vanishes.
Take $M^3$ to be the cylinder
$U \times {\bf R}$ in ${\bf R}^3 \times {\bf R} = {\bf R}^4$.  Then $M^3$ has three
distinct principal curvatures at each point, one of which is zero.  These are clearly constant along their
corresponding 1-dimensional curvature surfaces
(lines of curvature).

To get a proper Dupin hypersurface in ${\bf R}^5$ with three principal curvatures having respective multiplicities
$m_1 = m_2 = 1$, $m_3 =2$, one simply takes 
\begin{displaymath}
U \times {\bf R}^2 \subset {\bf R}^3 \times {\bf R}^2 = {\bf R}^5.
\end{displaymath}
for the set $U$ above.
To obtain a proper Dupin hypersurface $M^4$ in ${\bf R}^5$ with four principal curvatures, first invert the hypersurface
$M^3$ above in a 3-sphere in ${\bf R}^4$, chosen so that the image of $M^3$ contains an open subset $W^3$ on which no
principal curvature vanishes.  The hypersurface $W^3$ is proper Dupin, since the proper Dupin property is preserved by
M\"{o}bius transformations.  Now take $M^4$ to be the cylinder
$W \times {\bf R}$ in ${\bf R}^4 \times {\bf R} = {\bf R}^5$.
\end{proof}

In general, there are problems in trying to produce compact proper Dupin hypersurfaces by using these
constructions.  We now examine some of the problems involved with the the cylinder, 
surface of revolution, and tube constructions individually (see \cite[pp. 127--141]{Cec1} for more details).

For the cylinder construction,
the new principal curvature of $W^n$ is identically zero, while 
the other principal curvatures of $W^n$ are equal to those
of $M^{n-1}$.  Thus, if one of the principal curvatures $\mu$
of $M^{n-1}$ is zero at some points but not identically
zero, then the number of distinct principal curvatures 
is not constant on $W^n$, and so $W^n$ is Dupin but not proper Dupin.  

For the surface of revolution construction, if the focal point corresponding to a principal curvature
$\mu$ at a point $x$ of the profile submanifold $M^{n-1}$ lies on the axis of revolution ${\bf R}^{n-1}$, then
the principal curvature of $W^n$ at $x$ 
determined by $\mu$ is equal to the new principal curvature of $W^n$ resulting from
the surface of revolution construction.
Thus, if the focal point of $M^{n-1}$ corresponding to $\mu$ lies
on the axis of revolution for some but not all points of  $M^{n-1}$, then $W^n$ is not proper Dupin.

If $W^n$ is a tube in ${\bf R}^{n+1}$ of radius
$\varepsilon$ over $M^{n-1}$, then there are exactly two
distinct principal curvatures at the points in the set
$M^{n-1} \times \{ \pm \varepsilon \}$ in $W^n$, regardless
of the number of distinct principal curvatures on $M^{n-1}$.
Thus, $W^n$ is not a proper Dupin hypersurface unless the original hypersurface $M^{n-1}$ is totally umbilic, i.e., it
has only one distinct principal curvature at each point.

Another problem with these constructions is that they may not yield
an immersed hypersurface in ${\bf R}^{n+1}$. In the tube construction, if the
radius of the tube is the reciprocal of one of the 
principal curvatures of $M^{n-1}$ at some point, then the constructed object has a singularity.  For the
surface of revolution construction, a singularity occurs
if the profile submanifold $M^{n-1}$ intersects the axis of revolution. 

Many of the issues mentioned in the preceding paragraphs can be resolved by working
in the context of Lie sphere geometry and considering
Legendre lifts of hypersurfaces in Euclidean space (see \cite[pp. 127--148]{Cec1}).
In that context, a proper Dupin submanifold 
$\lambda : M^{n-1} \rightarrow \Lambda ^{2n-1}$ is said to be {\em reducible} if it is
is locally Lie equivalent to the Legendre lift of a hypersurface in ${\bf R}^n$ obtained by one of Pinkall's constructions.

Pinkall \cite{P4} found the following useful characterization of reducibility in the context of Lie sphere 
geometry.  For simplicity, we deal with the constructions as they are written at the beginning of this section, i.e., we take the case where the multiplicity of the new principal curvature is $m=1$.  Here we give Pinkall's proof
\cite[p. 438]{P4} (see also \cite[pp. 143--144]{Cec1}).

\begin{theorem}
\label{thm:4.2.9} 
A connected proper Dupin submanifold $\lambda: W^{n-1} \rightarrow \Lambda^{2n-1}$ is reducible if and only if there exists a curvature sphere $[K]$ of $\lambda$ that lies in a linear subspace of ${\bf P}^{n+2}$ of codimension two.
\end{theorem}

\begin{proof}
We first note that the following manifolds of spheres are hyperplane sections of the Lie quadric $Q^{n+1}$:\\

a) the hyperplanes in ${\bf R}^n$,

b) the spheres with a fixed signed radius $r$,

c) the spheres that are orthogonal to a fixed sphere.\\

To see this, we use the Lie coordinates given in equation (\ref{eq:1.3.4}).
In Case a), the hyperplanes are characterized by the equation $x_1 + x_2 = 0$, which clearly determines a hyperplane section of $Q^{n+1}$.  In Case b), the spheres with signed radius $r$ are determined by the linear equation
\begin{displaymath}
r(x_1 + x_2) = x_{n+3}.
\end{displaymath}
In Case c), it can be assumed that the fixed sphere is a hyperplane $H$ through the origin in ${\bf R}^n$.  A sphere is orthogonal to $H$ if and only if its center lies in $H$.  This clearly imposes a linear condition on the vector in 
equation (\ref{eq:1.3.4}) representing the sphere.

The sets a), b), c) are each of the form
\begin{displaymath}
\{x \in Q^{n+1} \mid \langle x, w \rangle = 0\},
\end{displaymath}
with $\langle w,w \rangle = 0, -1, 1$ in Cases a), b), c), respectively.

We can now see that every reducible Dupin hypersurface has a family of curvature spheres that is contained in two hyperplane sections of the Lie quadric as follows.

For the cylinder construction, the tangent hyperplanes of the cylinder are curvature spheres that are orthogonal to a fixed hyperplane in ${\bf R}^n$.  Thus, that family of curvature spheres is contained in a $n$-dimensional
linear subspace $E$ of  ${\bf P}^{n+2}$ such that the signature of $\langle \ ,\ \rangle$ on the polar subspace $E^\perp$ of $E$ is $(0,+)$.

For the surface of revolution construction, the new family of curvature spheres all have their centers in the axis of revolution, which is a linear subspace of codimension 2 in ${\bf R}^n$.  Thus, that family of curvature spheres is contained in a $n$-dimensional linear subspace $E$ of  ${\bf P}^{n+2}$ such that the signature of $\langle \ ,\ \rangle$ on the polar subspace $E^\perp$ of $E$ is $(+,+)$.

For the tube construction, the new family of curvature spheres all have the same radius, and their centers all lie in the hyperplane of ${\bf R}^n$ containing the manifold over which the tube is constructed.
Thus, that family of curvature spheres is contained in a $n$-dimensional linear subspace $E$ of  ${\bf P}^{n+2}$ such that the signature of $\langle \ ,\ \rangle$ on the polar subspace $E^\perp$ of $E$ is $(-,+)$.

Conversely, suppose that $K: W^{n-1} \rightarrow {\bf P}^{n+2}$ is a family of curvature spheres that is contained
in an $n$-dimensional linear subspace $E$ of  ${\bf P}^{n+2}$.  Then  $\langle \ ,\ \rangle$  must have signature
$(+,+)$, $(0,+)$ or $(-,+)$ on the polar subspace $E^\perp$, because otherwise $E \cap Q^{n+1}$ would be empty or would consist of a single point.

If the signature of $E^\perp$ is $(+,+)$, then there exists a Lie sphere transformation $A$ which takes  $E$ 
to a space  $F = A(E)$ such that $F \cap Q^{n+1}$ consists of all spheres that have their centers in a fixed $(n-2)$-dimensional linear subspace ${\bf R}^{n-2}$ of 
${\bf R}^n$.  Since one family of curvature spheres of this Dupin submanifold $A\lambda$ lies in $F \cap Q^{n+1}$, and the Dupin submanifold $A\lambda$ is the envelope of these spheres,  $A\lambda$ must be a surface of revolution with the axis ${\bf R}^{n-2}$ (see \cite[pp. 142--143]{Cec1} for more detail on envelopes of
families of spheres in this situation), and so $\lambda$ is reducible.

If the signature of $E^\perp$ is $(0,+)$, then there exists a Lie sphere transformation $A$ which takes  $E$ 
to a space  $F = A(E)$ such that $F \cap Q^{n+1}$ consists of hyperplanes orthogonal to a fixed hyperplane in 
${\bf R}^n$.  Since one family of curvature spheres of this Dupin submanifold $A\lambda$ lies in $F \cap Q^{n+1}$, and the Dupin submanifold $A\lambda$ is the envelope of these spheres,  $A\lambda$ is obtained as a result of the cylinder construction, and so $\lambda$ is reducible.

If the signature of $E^\perp$ is $(-,+)$, then there exists a Lie sphere transformation $A$ which takes  $E$ 
to a space  $F = A(E)$ such that $F \cap Q^{n+1}$ 
consists of spheres that all have the same radius and whose centers lie in a hyperplane
${\bf R}^{n-1}$ of ${\bf R}^n$.  Since one family of curvature spheres of this Dupin submanifold $A\lambda$ lies in $F \cap Q^{n+1}$, and the Dupin submanifold $A\lambda$ is the envelope of these spheres,  $A\lambda$ is obtained as a result of the tube construction, and so $\lambda$ is reducible.
\end{proof}

\begin{remark}
\label{cone-construction}
{\rm Note that for the cone construction $(4)$ at the beginning of this section, the new family $[K]$ of curvature spheres
consists of hyperplanes through the origin that are tangent to the cone along the rulings.  
In the Lie coordinates (\ref{eq:1.3.4}),
the origin corresponds to the point $[e_1 + e_2]$,  while the hyperplanes are orthogonal to the improper point
$[e_1 - e_2]$.  Thus, the hyperplanes through the origin correspond by equation (\ref{eq:1.3.4})
to points in the linear subspace $E$, whose orthogonal complement $E^\perp$ is spanned $\{e_1 + e_2, e_1 - e_2\}$. This space $E^\perp$ is also spanned by $\{e_1, e_2\}$, and so the signature of $E^\perp$ is $(-,+)$, the same as for the tube construction. Therefore, the cone construction and the tube construction are Lie equivalent.
(See Remark 5.13 of \cite[p. 144]{Cec1} for more detail.)
Finally, there is one more geometric interpretation of the tube
construction.  Note that a family $[K]$ of curvature spheres that lies in a linear subspace whose orthogonal
complement has signature $(-,+)$ can also be considered to consist of spheres in $S^n$ of constant radius
in the spherical metric
whose centers lie in a hyperplane.  The corresponding proper Dupin submanifold can thus be
considered to be a tube in the spherical metric over a lower-dimensional submanifold 
that lies in a hyperplane section of $S^n$.}
\end{remark}

As we noted after the proof of Theorem \ref{thm:4.1.1}, there are difficulties in constructing compact proper
Dupin hypersurfaces by using Pinkall's constructions.
We can construct a reducible compact proper Dupin hypersurface with two principal curvatures 
by revolving a circle $C$ in ${\bf R}^3$ about an axis ${\bf R}^1 \subset {\bf R}^3$ that is disjoint from $C$ 
to obtain a torus of revolution.  Of course, this can be generalized to higher dimensions, as in
Theorem \ref{thm:4.3.2},  by revolving a $q$-sphere 
$S^q \subset {\bf R}^{q+1} \subset {\bf R}^n$ about an axis ${\bf R}^q \subset {\bf R}^{q+1}$ to
obtain a compact cyclide of Dupin of characteristic $(p,q)$, where $p + q = n-1$.  Such a cyclide has two
principal curvatures at each point having respective multiplicities $p$ and $q$.

However, Cecil, Chi and Jensen \cite{CCJ2} (see also \cite[pp. 146--147]{Cec1}) showed that 
every compact proper Dupin hypersurface with more than two principal curvatures is irreducible, as stated in the following theorem.
\noindent
\begin{theorem}
\label{thm:CCJ-2007}
(Cecil-Chi-Jensen, 2007)
If $M^{n-1} \subset {\bf R}^n$ is a compact, connected proper Dupin hypersurface with $g \geq 3$
principal curvatures, then $M^{n-1}$ is irreducible.
\end{theorem}

The proof uses known facts about the topology of a compact proper Dupin hypersurface
and the topology of a compact hypersurface obtained by one of Pinkall's constructions
(see \cite{CCJ2} or \cite[pp. 146--148]{Cec1} for a complete proof).\\

\section{Classifications of Dupin Hypersurfaces}
\label{sec:12}
In this section, we discuss classification results concerning proper Dupin hypersurfaces in ${\bf R}^n$ or $S^n$ that have been obtained using the techniques of Lie sphere geometry.  These primarily concern two important classes: compact proper Dupin hypersurfaces, and irreducible proper Dupin hypersurfaces.  Of course,
Theorem \ref{thm:CCJ-2007} shows there is a strong connection between these two classes of hypersurfaces, and many classifications of compact proper Dupin hypersurfaces with $g \geq 3$
principal curvatures have been obtained by assuming that the hypersurface is irreducible and 
working locally in the context of Lie sphere geometry using the method of moving frames.  
(See, for example, the papers of Pinkall \cite{P1}, \cite{P3}--\cite{P4}, Cecil and Chern \cite{CC2}, Cecil and Jensen \cite{CJ2}--\cite{CJ3}, and Cecil, Chi and Jensen \cite{CCJ2}.)  See also \cite{Cec10} for a survey of
classifications of compact proper Dupin hypersurfaces.

Two key tools in many of these classifications are:\\

\noindent
1) the Lie sphere geometric characterization of Legendre lifts of isoparametric hypersurfaces given in 
Theorem \ref{thm:3.5.6},\\

\noindent
2) Pinkall's characterization of reducible proper Dupin hypersurfaces given in Theorem \ref{thm:4.2.9}.\\

\noindent
We now summarize these classifications and give references to their proofs.\\

We begin by recalling some important facts about compact proper Dupin hypersurfaces embedded in $S^n$.  
Following M\"{u}nzner's work \cite{Mu}--\cite{Mu2} on isoparametric hypersurfaces, Thorbergsson \cite{Th1} proved the following theorem which shows that M\"{u}nzner's restriction on the number $g$ of 
distinct principal curvatures of an isoparametric hypersurface also holds for compact proper Dupin hypersurfaces embedded in $S^n$.  This is in stark contrast to Pinkall's Theorem \ref{thm:4.1.1} which states that
there are no restrictions on the number of distinct principal curvatures or their multiplicities for
non-compact proper Dupin hypersurfaces.

\begin{theorem}
\label{thm:thorbergsson}
(Thorbergsson, 1983) The number $g$ of distinct principal curvatures of a compact, connected
proper Dupin hypersurface
$M \subset S^n$ must be $1,2,3,4$ or $6$.
\end{theorem}

In proving this theorem,
Thorbergsson first shows that a compact, connected proper Dupin hypersurface $M \subset S^n$
must be tautly embedded,
that is, every nondegenerate spherical distance function $L_p (x) = d(p,x)^2$,  for  $p \in S^n$,
has the minimum number of critical points required by the Morse inequalities on $M$.
Thorbergsson then uses the fact that $M$ is tautly embedded in $S^n$ to show that $M$ divides $S^n$ into
two ball bundles over the first focal submanifolds, $M_+$ and $M_-$, on either side of $M$ in $S^n$.
This gives the same topological situation as in the isoparametric case,
and the theorem then follows from M\"{u}nzner's  \cite{Mu2} proof of the restriction on $g$ for isoparametric
hypersurfaces.

The topological situation that $M$ divides $S^n$ into
two ball bundles over the first focal submanifolds, $M_+$ and $M_-$, on either side of $M$ in $S^n$ also
leads to important restrictions on the multiplicities of the principal 
curvatures of compact proper Dupin hypersurfaces, due to
Stolz \cite{Stolz} for $g=4$, and to Grove and Halperin \cite{GH} for $g=6$.
These restrictions were obtained by using advanced topological considerations in each case, and they show that 
the multiplicities of the principal curvatures of
a compact proper Dupin hypersurface embedded in $S^n$ must be the same as the multiplicities 
of the principal curvatures of some isoparametric hypersurface in $S^n$.

Grove and Halperin \cite{GH} also gave a list of the integral homology of all compact
proper Dupin hypersurfaces, and Fang \cite{Fang2} found results on the topology of compact proper Dupin hypersurfaces with $g=6$ principal curvatures.

In 1985, it was known that every compact, connected proper Dupin hypersurface 
$M \subset S^n$ (or ${\bf R}^n$) with
$g= 1,2$ or 3 principal curvatures is Lie equivalent to an isoparametric hypersurface in $S^n$.
At that time, every other known example
of a compact, connected proper Dupin hypersurface in $S^n$ 
was also Lie equivalent to an isoparametric hypersurface in $S^n$. 
This together with Thorbergsson's Theorem \ref{thm:thorbergsson} above led to the following
conjecture by Cecil and Ryan \cite[p. 184]{CR7} (which we have rephrased slightly).

\begin{conjecture} 
\label{cecil-ryan} 
(Cecil-Ryan, 1985) Every compact, connected proper 
Dupin hypersurface $M \subset S^n$ $($or ${\bf R}^n)$
is Lie equivalent to an isoparametric hypersurface in $S^n$.
\end{conjecture}

We now discuss the state of the conjecture for each of the values of $g$. The case $g = 1$ is simply the case of 
totally umbilic hypersurfaces, and $M$ is a great or small hypersphere in $S^n$.  
In the case $g=2$, Cecil and Ryan \cite{CR2} showed that $M$
is a cyclide of Dupin (see Section \ref{sec:10}), and thus it is M\"{o}bius equivalent to a standard product of spheres
\begin{displaymath}
S^p (r) \times S^{n-1-p} (s) \subset S^n (1) \subset {\bf R}^{n+1}, \quad r^2 + s^2 = 1,
\end{displaymath}
which is an isoparametric hypersurface.

In the case $g=3$, Miyaoka \cite{Mi1} proved that $M$ is Lie equivalent 
to an isoparametric hypersurface
(see also Cecil-Chi-Jensen \cite{CCJ2} for a different proof using the fact that
compactness implies irreducibility, i.e., Theorem \ref{thm:CCJ-2007}).
Earlier, Cartan \cite{Car3} had shown that an isoparametric hypersurface with $g=3$
principal curvatures is a tube over a standard
embedding of a projective plane ${\bf FP}^2$,
for ${\bf F} = {\bf R}, {\bf C}, {\bf H}$ (quaternions) or ${\bf O}$ (Cayley numbers) 
in $S^4, S^7, S^{13}$
and $S^{25}$, respectively. For ${\bf F} = {\bf R}$, a standard embedding is a spherical 
Veronese surface (see also \cite[pp. 151--155]{CR8}). 

All attempts to verify Conjecture \ref{cecil-ryan} in the cases $g=4$ and 6 were unsuccessful, however.  Finally, in 1988,
Pinkall and Thorbergsson \cite{PT1} and Miyaoka and Ozawa \cite{MO} 
gave two different methods for producing
counterexamples to Conjecture \ref{cecil-ryan}  with $g =4$ principal curvatures.  
The method of Miyaoka and Ozawa also yields
counterexamples to the conjecture in the case $g=6$.  

These examples were shown to be counterexamples 
to the conjecture by a consideration of their Lie curvatures, which were introduced by Miyaoka \cite{Mi3}.
Lie curvatures are cross-ratios of the principal curvatures taken four at a time, and they are equal to the cross-ratios
of the corresponding curvature spheres along a projective line by Theorems \ref{thm:3.5.1}
and  \ref{thm:3.5.2}.  Since Lie sphere transformations map curvature spheres to curvature spheres
by Theorem \ref{thm:3.4.3}, and 
they preserve cross-ratios of four points along a projective line (since they are projective transformations),
Lie curvatures are invariant under Lie sphere transformations.  Obviously, the Lie curvatures must
be constant for a Legendre submanifold that is Lie equivalent to the Legendre lift of an isoparametric
hypersurface in a sphere.

The examples of Pinkall and Thorbergsson  are obtained by 
taking certain deformations of the isoparametric hypersurfaces
of FKM-type constructed by  Ferus, Karcher and M\"{u}nzner \cite{FKM} using representations of Clifford algebras.  
Pinkall and Thorbergsson proved that their examples are not Lie
equivalent to an isoparametric hypersurface by showing that the Lie curvature does not have the
constant value $\psi = 1/2$, 
as required for a hypersurface with $g=4$ that is Lie equivalent 
to an isoparametric hypersurface (if the principal curvatures are appropriately ordered).  
Using their methods, one can also show directly
that the Lie curvature is not constant for their examples (see \cite[pp. 309--314]{CR8}).

The construction of counterexamples to  Conjecture \ref{cecil-ryan} due to Miyaoka and Ozawa \cite{MO} (see also
\cite[pp. 117--123]{Cec1}) is based on the Hopf fibration
$h:S^7 \rightarrow S^4$.  
Miyaoka and Ozawa show that if $W^3$ is a proper Dupin hypersurface in
$S^4$ with $g$ distinct principal curvatures, then
$M = h^{-1}(W^3)$ is a proper Dupin hypersurface in $S^7$ with $2g$ principal
curvatures.  Next they show that if a compact, connected 
hypersurface $W^3 \subset S^4$ is proper Dupin
but not isoparametric, then the Lie curvatures of
$h^{-1}(W^3)$ are not constant, and therefore $h^{-1}(W^3)$ is not Lie
equivalent to an isoparametric hypersurface in $S^7$.  For $g=2$ or 3, this gives
a compact proper Dupin hypersurface $h^{-1}(W^3)$ in $S^7$ with $g=4$ or 6, respectively,
that is not Lie equivalent to an isoparametric hypersurface.

As noted above, all of these hypersurfaces are shown to be
counterexamples to Conjecture \ref{cecil-ryan} by proving that they do not have
constant Lie curvatures.  This led to a revision of Conjecture \ref{cecil-ryan} by Cecil, Chi and Jensen
\cite[p. 52]{CCJ4} in 2007 that contains the additional assumption of constant Lie curvatures.
This revised conjecture
is still an open problem, although it  has been shown to be true in some cases,
which we will describe after stating the conjecture.

\begin{conjecture}
\label{revised-conjecture}
(Cecil-Chi-Jensen, 2007)
Every compact, connected proper Dupin hypersurface in $S^n$ with four or six principal curvatures
and constant Lie curvatures is Lie equivalent to an isoparametric hypersurface in $S^n$.
\end{conjecture}

We first note that in 1989, Miyaoka \cite{Mi3}--\cite{Mi4} showed that if some additional assumptions are made regarding the intersections of the leaves of the various principal foliations, then this revised conjecture is true in both cases $g=4$ and 6.  
Thus far, however, it has not been proven that Miyaoka's additional assumptions are satisfied in general.

Cecil, Chi and Jensen \cite{CCJ2} 
made progress on the revised conjecture in the case $g=4$ by using the
fact that compactness implies irreducibility for a proper Dupin hypersurface
with $g \geq 3$ (see Theorem \ref{thm:CCJ-2007}), 
and then working locally with irreducible proper hypersurfaces in the context 
of Lie sphere geometry.  

If we fix the order of the principal curvatures of $M$ to be,
\begin{equation}
\label{eq:pc-order}
\mu_1 < \mu_2 < \mu_3 < \mu_4,
\end{equation}
then there is only one Lie curvature,
\begin{equation}
\label{eq:lie-curv}
\psi = \frac{(\mu_1 -\mu_2)(\mu_4 - \mu_3)}{(\mu_1 -\mu_3)(\mu_4 - \mu_2)}.
\end{equation}

For an isoparametric
hypersurface with four principal curvatures ordered as in equation (\ref{eq:pc-order}),
M\"{u}nzner's results \cite{Mu}--\cite{Mu2} imply that
the Lie curvature $\psi = 1/2$, and
the multiplicities satisfy $m_1 = m_3$, $m_2 = m_4$.
Furthermore, if $M \subset S^n$ is a compact, connected proper Dupin hypersurface 
with $g=4$, then the multiplicities of the principal curvatures must be the same as those of an isoparametric hypersurface by the work of Stolz \cite{Stolz}, so they satisfy $m_1 = m_3$, $m_2 = m_4$.

Cecil-Chi-Jensen \cite{CCJ2} proved the following 
local classification of irreducible proper Dupin hypersurfaces with four principal curvatures and constant Lie curvature 
$\psi = 1/2$. In the case where all the multiplicities equal one, this theorem was first proven
by Cecil and Jensen \cite{CJ3}.

\begin{theorem}
\label{CCJ} 
(Cecil-Chi-Jensen, 2007)
Let $M \subset S^n$ be a connected irreducible proper Dupin hypersurface with four principal curvatures 
ordered as in equation (\ref{eq:pc-order}) having multiplicities,
\begin{equation}
\label{eq:restricted}
m_1 = m_3 \geq 1, \quad m_2 = m_4 =1,
\end{equation}
and constant Lie curvature $\psi = 1/2$. Then $M$ is Lie equivalent to an
isoparametric hypersurface in $S^n$.
\end{theorem}

Key elements in the proof of Theorem \ref{CCJ} are the Lie geometric criteria for reducibility (Theorem \ref{thm:4.2.9})
due to Pinkall \cite{P4}, and the criterion for Lie equivalence to
an isoparametric hypersurface (Theorem \ref{thm:3.5.6}).

By Theorem \ref{thm:CCJ-2007} above, we know that
compactness implies irreducibility for proper Dupin hypersurfaces with more than two principal curvatures.
Furthermore, Miyaoka \cite{Mi3} proved that if $\psi$ is constant on a compact proper Dupin hypersurface $M \subset S^n$ with $g=4$, then $\psi = 1/2$ on $M$, when
the principal curvatures are ordered as in equation (\ref{eq:pc-order}).  
As a consequence, we get the following corollary of
Theorem \ref{CCJ}.

\begin{corollary}
Let $M \subset S^n$ be a compact, connected proper Dupin hypersurface with 
four principal curvatures having multiplicities
\begin{displaymath}
m_1 = m_3 \geq 1, \quad m_2 = m_4 =1,
\end{displaymath}
and constant Lie curvature $\psi$.  Then $M$ is Lie equivalent to an isoparametric hypersurface in $S^n$.
\end{corollary}

The remaining open question is what happens if $m_2$ is also allowed to be greater than one, i.e.,
\begin{equation}
\label{eq:general}
m_1 = m_3 \geq 1,\quad  \ m_2 = m_4 \geq 1,
\end{equation}
and constant Lie curvature $\psi$.

Regarding this question, we note that
the local proof of Theorem \ref{CCJ} of Cecil, Chi and Jensen \cite{CCJ2}
uses the method of moving frames, and it involves a large system
of equations that contains certain sums if some $m_i$ is greater than one, but no 
corresponding sums if all $m_i$ equal one.  These sums make the calculations significantly more difficult, and
so far this method has not led to a proof in the general case (\ref{eq:general}).  Even so, this approach to
proving Conjecture \ref{revised-conjecture} in the case $g=4$
could possibly be successful with some additional insight regarding the
structure of the calculations involved.

Finally, Grove and Halperin \cite{GH} proved in 1987 that if $M \subset S^n$ is a compact proper Dupin hypersurface
with $g=6$ principal curvatures, then all the principal curvatures must have the same multiplicity $m$, and
$m = 1$ or 2.  This was shown earlier for isoparametric hypersurfaces with $g=6$ by Abresch \cite{Ab}.
Grove and Halperin also proved other topological 
results about compact proper Dupin hypersurfaces that support Conjecture \ref{revised-conjecture} in the case $g=6$.

As mentioned above,
Miyaoka \cite{Mi4} showed that if some additional assumptions are made regarding the intersections of the leaves of the various principal foliations, then Conjecture \ref{revised-conjecture} is true in the case $g=6$.  
However, it has not been proven that Miyaoka's additional assumptions are satisfied in general,
and so Conjecture \ref{revised-conjecture} remains as an open problem in the case $g=6$.

\noindent Thomas E. Cecil

\noindent Department of Mathematics and Computer Science

\noindent College of the Holy Cross, 

\noindent Worcester, MA 01610, U.S.A.

\noindent email: tcecil@holycross.edu

\end{document}